\documentclass[oneside,11pt]{amsart}
\usepackage{amssymb, amsmath, amsthm, enumerate,graphicx}
\usepackage{hyperref,tikz}
\usepackage{thmtools, comment,color,todonotes, cleveref}
\usetikzlibrary{decorations.markings}

\DeclareMathOperator{\cat}{CAT}

\newtheorem{innercustomthm}{Theorem}
\newenvironment{customthm}[1]
{\renewcommand\theinnercustomthm{#1}\innercustomthm}
{\endinnercustomthm}

\newtheorem{thm}{Theorem}[section]
\newtheorem{proposition}[thm]{Proposition}
\newtheorem{lem}[thm]{Lemma}

\newtheorem{corollary}[thm]{Corollary}
\newtheorem*{thm*}{Theorem}

\newtheorem{prop}[thm]{Proposition}

\theoremstyle{definition}
\newtheorem{defn}[thm]{Definition}
\newtheorem{example}[thm]{Example}

\newtheorem{remark}[thm]{Remark}
\newtheorem{question}[thm]{Question}
\newtheorem*{Outline}{Outline}
\newtheorem*{ack}{Acknowledgements}

\newcommand{\B}{\mathcal{B}}
\newcommand{\W}{\mathcal{W}}
\renewcommand{\P}{\mathcal{P}}
\newcommand{\bP}{\mathbb{P}}
\newcommand{\A}{\mathcal{A}}
\newcommand{\G}{\mathcal{G}}
\newcommand{\X}{\overline{X}}
\newcommand{\N}{\mathbb{N}}
\newcommand{\R}{\mathbb{R}}
\newcommand{\Z}{\mathbb{Z}}
\renewcommand{\H}{\mathbb{H}}
\renewcommand{\R}{\mathcal{R}}
\newcommand{\F}{\mathcal{F}}
\newcommand{\D}{\mathcal{D}}
\newcommand{\C}{\mathcal{C}}
\newcommand{\E}{\mathbb{E}}

\newcommand{\T}{\mathcal{T}}
\newcommand{\bR}{\mathbb{bR}}
\renewcommand{\epsilon}{\varepsilon}

\newcommand{\bnd}{\partial}

\newcommand{\inverse}{^{-1}}
\newcommand{\lonely}{\partial_{I_L}}

\DeclareMathOperator{\Itin}{Itin}
\DeclareMathOperator{\Nerve}{Nerve}

\title{Boundaries of groups with isolated flats are path connected}
\date{}
\author{Michael Ben-Zvi}

\begin{document}
	
	\begin{abstract}
	A seminal result in geometric group theory is that a 1-ended hyperbolic group has a locally connected visual boundary. As a consequence, a 1-ended hyperbolic group also has a path connected visual boundary. In this paper, we study when this phenomenon occurs for CAT(0) groups. We show if a 1-ended CAT(0) group with isolated flats acts geometrically on a CAT(0) space, then the visual boundary of the space is path connected. As a corollary, we prove all CAT(0) groups with isolated flats are semistable at infinity.
	
	\end{abstract}

\maketitle
	\section{Introduction}
	
	We study $\cat(0)$ groups and the boundaries of the spaces on which they act. A group is $\cat(0)$ if it acts geometrically (properly discontinuously, cocompactly, and by isometries) on some $\cat(0)$ space. The motivating question for this article is the following:

\begin{question} \label{question: pconn}
	If $G$ acts geometrically on a 1-ended $\cat(0)$ space $X$, when is the visual boundary, $\bnd X$, path connected?
\end{question}

There are two main reasons to investigate this property for boundaries of $\cat(0)$ spaces. The first comes from the theory of hyperbolic groups. All 1-ended hyperbolic groups have locally connected boundaries  \cite{BestvinaMess,Swarup,BowCutPoints}. Some natural examples of non-hyperbolic $\cat(0)$ groups, such as $F_2\times \Z$, show that this phenomenon does not occur in general. However, a boundary which is connected and locally connected is also globally path connected. In this sense, path connectivity is the `next best thing' one can hope for in the boundary of a $\cat(0)$ group. At the same time, there are known examples of $\cat(0)$ groups which act on spaces with non-path connected visual boundaries (see Example \ref{example: Croke Kleiner} for one). Therefore, we want to find conditions on $G$ (or on $X$) which determine when $\bnd X$ is path connected.

Another reason to be interested in path connectivity is due to its relationship with semistability. A proper, 1-ended $\cat(0)$ space $X$ is \textit{semistable at infinity} if any two geodesic rays are properly homotopic. Being semistable at infinity is a quasi-isometry invarient, so showing a group is semistable at infinity comes down to finding an appropriate space on which this group acts. Geoghegan  conjectured that all $\cat(0)$ groups are semistable at infinity (in fact, he has conjectured that all finitely presented groups are semistable at infinity). Geoghegan also shows that given a $\cat(0)$ space $X$ if $\bnd X$ is path connected, then $X$ is semistable at infinity \cite{GeoghPconnSemistable}. Piecing all of this together, knowing that $G$ acts geometrically on a $\cat(0)$ space with a path connected visual boundary shows that $G$ is semistable at infinity.

The majority of this paper is dedicated to answering Question \ref{question: pconn} for a class of $\cat(0)$ groups which exhibit more hyperbolicity than others: $\cat(0)$ groups with isolated flats.  These groups are hyperbolic relative to a collection of flat stabilizers and therefore share many properties with hyperbolic groups. Hruska and Ruane have found necessary and sufficient conditions for a  1-ended $\cat(0)$ group with isolated flats to have a locally connected boundary \cite{HR17}. There are many examples which do not have locally connected boundaries, see Section \ref{section: examples} for a few. We therefore study the general case. In this paper, we prove the following:

\begin{customthm}{\ref{thm main}}Let $G$ be a 1-ended $\cat(0)$ group with isolated flats acting geometrically on $X$. Then $\bnd X$ is path connected.
\end{customthm} 

As an immediate corollary (via the results in \cite{GeoghPconnSemistable}) we have the following:

\begin{corollary}
	Let $G$ be a 1-ended $\cat(0)$ group with isolated flats. Then $G$ is semistable at infinity.
\end{corollary}
Other authors have some results which overlap with this corollary. For example, Mihalik and Swenson  show that many 1-ended relatively hyperbolic groups are semistable at infinity \cite{MihalikSwensonSemistable}. However, they do not cover all cases. They assume that their relatively hyperbolic groups have a trivial maximal peripheral splitting and many $\cat(0)$ groups with isolated flats have non-trivial maximal peripheral splittings. Hruska and Ruane  show that most groups which are hyperbolic relative to polycyclic groups are semistable at infinity \cite{HruskaRuaneSemistable}. They assume that these groups have no non-central elements of order 2. We require no assumptions on the $\cat(0)$ groups with isolated flats to conclude it is semistable at infinity.

Along the way to proving Theorem \ref{thm main} we introduce a combination theorem for $\cat(0)$ spaces and their boundaries:

\begin{customthm}{\ref{combo}}

		Let $X$ be a $\cat(0)$ space and $\B$ a block decomposition. Suppose the following conditions hold:
	
	\begin{enumerate}
		\item for each $B\in \B$, $\bnd B$ is path connected,
		\item for each $W\in \W$, $\bnd W$ is non-empty, and
		\item all irrational rays are lonely.
	\end{enumerate}
	Then $\bnd X$ is path connected.
\end{customthm}

For exact definitions of the terms in Theorem \ref{combo}, see Section \ref{section block decomp}. The example to keep in mind is when we have a $\cat(0)$ group $G$ which is an amalgamated product $G=A\ast_C B$. In many cases of interest, there is a natural `tree-of-spaces' $X$ on which $G$ acts geometrically and then Theorem \ref{combo} essentially says if (1) $A$ and $B$ act on spaces $X_A$ and $X_B$ with path connected boundaries, (2) $\bnd X$ is connected,  and (3) a technical condition related to how the boundaries of $X_A$ and $X_B$ interact in $X$, then $\bnd X$ is path connected. Amalgamating $\cat(0)$ groups is the primary method for forming new $\cat(0)$ groups. In order to fully understand what types of spaces can appear as the boundary of a $\cat(0)$ group, we need to understand properties of the boundaries of amalgams.

\begin{Outline}
	In Section \ref{section: examples}, we describe three key examples that we return to throughout the paper. Sections \ref{section visual boundary}, \ref{rel hyp section}, and \ref{section:isolated flats} provide the necessary background information on the visual boundary, relatively hyperbolic groups, and $\cat(0)$ groups with isolated flats. We formally define a block decompositions and prove Theorem \ref{combo} in Section \ref{section block decomp}. 
	
	The remaining sections build up to the proof of Theorem \ref{thm main}. In Section \ref{section:fixing a space}, for a given 1-ended $\cat(0)$ group $G$ with isolated flats, we use the group's maximal peripheral splitting (introduced by Bowditch in \cite{Bow01}), Hruska-Ruane's convex splitting theorem \cite{HR17}, and Bridson-Haefliger's Equivariant Gluing Theorem \cite{BH99} to fix a space $X$ on which $G$ acts geometrically. Hruska-Kleiner's work \cite{HK05} shows that if $G$ is $\cat(0)$ with isolated flats, then the boundary is well-defined. Therefore fixing a space on which $G$ acts does not affect the boundary. We then focus on a collection of subsets $Y_v\subset X$ which cover $X$. These subsets are defined at the end of Section \ref{section:fixing a space}. In Sections \ref{section:decomp}, \ref{section:proper and cocompact}, and \ref{section:local conn at flats} we show that $\bnd Y_v$ is path connected. Applying  Theorem \ref{combo} finishes the proof of Theorem \ref{thm main}. 
	
	Of potentially independent interest is the work in Section \ref{section:proper and cocompact}. Here, we establish a connection between neighborhoods of points in $\overline{F}$, the closure of a flat, and neighborhoods in  $\bnd Y_v \setminus F$. This work was done by Haulmark in the locally connected setting \cite{Ha17} and we expand the results to our setting.
	
\end{Outline}

\begin{ack}
	I would like to thank Kim Ruane for her help and guidance throughout this project. I would also like to thank Genevieve Walsh, Mike Mihalik, and Rob Kropholler for many interesting discussions and for feedback on my drafts.
\end{ack}

\section{Key Examples} \label{section: examples}

In this section, we outline three key examples that capture the difficulty in this work. Two are surface group amalgams and the third is a right-angled Coxeter group which is an amalgamation of groups which are virtually surface groups. All three examples are $\cat(0)$ groups with isolated flats and the boundaries of the spaces on which they act are not locally connected.

\begin{figure}[h]
	\includegraphics[width=.4\textwidth]{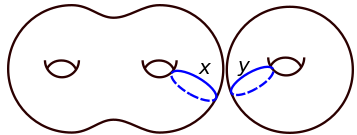}
	\centering
	\caption{A torus and a genus 2 surface with the identification $x=y$.}
	\label{figure: s2 torus}
\end{figure}

\begin{example} \label{example: s2 torus}
	Let $\tilde{X_1}$ be the surface amalgam seen in Figure \ref{figure: s2 torus}, $G_1$ its fundamental group, and $X_1$ its universal cover. This group is $\cat(0)$ with isolated flats, with the flats coming from the torus. $X_1$ consists of Euclidean and hyperbolic planes fitting together in a tree-like way. The boundary, $\bnd X_1$, is made up of circles coming from these two types of planes along with the boundary points of rays which pass through infinitely many such planes. It is clear that the union of the circles is path connected since adjacent planes share a pair of boundary points. Because of the isolated flats, if two based rays pass through the same infinite collection of planes, then the rays are asymptotic. Applying Theorem \ref{combo}, we can conclude that $\bnd X_1$ is path connected.
\end{example}

\begin{figure}[h]
	\centering
	\includegraphics[width=.4\textwidth]{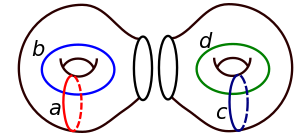}
	\caption{Two tori, each with a boundary component, with the identification $[a,b]^2=[c,d]^2$}
	\label{figure:tori with boundary}
	
\end{figure}

\begin{example}\label{example: punctured tori}
	Let $\tilde{X_2}$ be 2 tori with boundary components as identified in Figure \ref{figure:tori with boundary} and let $G_2$ be the fundamental group. By identifying $([a,b])^2$ with $([c,d])^2$, we create a Klein bottle group in $G$ generated by $[a,b]$ and $[c,d]$. This subgroup (and its conjugates) are the peripheral subgroups for the relatively hyperbolic structure on $G_2$. Since the Klein bottle group is virtually $\Z^2$, this makes $G_2$ $\cat(0)$ with isolated flats and with maximal peripheral splitting seen in Figure \ref{fig: max periph splitting G2}.
	
		\begin{figure}[h]
			\centering
		\begin{tikzpicture}
	
	\node[below] at (0,0) (A) {$K=\langle[a,b],[c,d]\rangle$};
	\node[below] at (-4,0) (B) {$F_2=\langle a,b\rangle$};
	\node[below] at (4,0) (C) {$F_2=\langle c,d \rangle$};
	\node[above] at (2,0) {$\Z=\langle [c,d] \rangle$};
	\node[above] at (-2,0) {$\Z=\langle [a,b]\rangle $};

	\draw[fill] (0,0) circle [radius=0.05];
	\draw[fill] (4,0) circle [radius=0.05];
	\draw[fill] (-4,0) circle [radius=0.05];
	\draw (-4,0) -- (0,0) -- (4,0);
	\end{tikzpicture}
	\caption{The maximal peripheral splitting of $G_2$}
	\label{fig: max periph splitting G2}
\end{figure}
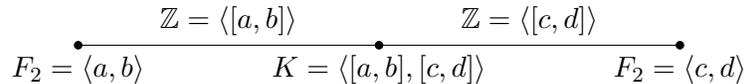
A $\cat(0)$ space $X_2$ on which $G_2$ acts geometrically can be constructed using Bridson and Haefliger's Equivariant Gluing Theorem \cite{BH99}. This construction captures the isolated flats explicitly since $X_2$ consists of Euclidean planes coming from the Klein bottle group and truncated hyperbolic planes coming from the free groups. Once again, these spaces fit together in a tree-like way and the boundary $\bnd X_2$ is made up of boundary points from the truncated hyperbolic planes, the Euclidean planes, and the rays which pass through infinitely many such spaces.

To see that $\bnd X_2$ is path connected, consider a truncated hyperbolic plane union every flat which shares a boundary point with it. Call this space $Y$. This is equivalent to taking a $F_2$-vertex $v$ of the Bass-Serre tree for the splitting and letting $Y$ be the of the vertex space for $v$ union the vertex spaces for each of the vertices of the Bass-Serre tree adjacent to $v$. It is clear that $X_2$ is the union of all such $Y$. To see that $\bnd Y$ is path connected, notice that it is a Cantor set coming from the truncated hyperbolic planes union a collection of circles coming from the Euclidean planes. These circle `fill in' the gaps between points of the Cantor set, essentially replacing the missing intervals with circles. Therefore $\bnd Y$ is a circle with a collection of extra arcs attached and is path connected. 

Since $\bnd Y_v$ is path connected for each $v$ in the Bass-Serre tree and $X$ is the union of all such $Y_v$, then the only points to check are those which come from rays traveling through infinitely many truncated hyperbolic and Euclidean plans. Once again, the isolated flats condition covers this and we can apply Theorem \ref{combo} to conclude $\bnd X_2$ is path connected.
\end{example}

\begin{figure}[h]
	\centering
	\begin{tikzpicture}
	\newdimen\r
	\r=1cm
	\draw (0:\r)--(60:\r)--(120:\r)--(180:\r)--(240:\r)--(300:\r)--(360:\r);
	\draw (330:1.7)--(300:\r)--(0:0)--(60:\r)--(30:1.7)--(0:2)--(330:1.7);
	
	\draw[fill] (0:\r) circle [radius=0.05];
	\draw[fill] (60:\r) circle [radius=0.05];
	\draw[fill] (120:\r) circle [radius=0.05];
	\draw[fill] (180:\r) circle [radius=0.05];
	\draw[fill] (240:\r) circle [radius=0.05];
	\draw[fill] (300:\r) circle [radius=0.05];
	\draw[fill] (360:\r) circle [radius=0.05];
	\draw[fill] (0:0) circle [radius=0.05];
	\draw[fill] (330:1.7) circle [radius=0.05];
	\draw[fill] (30:1.7) circle [radius=0.05];
	\draw[fill] (0:2) circle [radius=0.05];

	\end{tikzpicture}
	\caption{A Right-angled Coxeter group consisting of two hexagons sharing a pair of vertices.}
	\label{figure:RACG}
\end{figure}
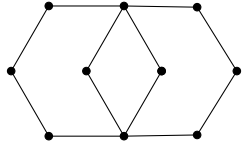

\begin{example}\label{example: RACG}
Let $G_3$ be the Right-angled Coxeter group in Figure \ref{figure:RACG} and let $X$ be the Davis complex for $G_3$. This group is $\cat(0)$ with isolated flats, with the flats coming from the square in the graph. Let $H$ be the RACG whose defining graph is a hexagon, then $G_3$ has the splitting seen in Figure \ref{figure: $G_3$ splitting}.

\begin{figure}[h]
	\centering
	
	\begin{tikzpicture}
	\node[below] at (-2,0) {$H$};
	\node[below] at (0,0) {$(\Z_2\ast \Z_2)^2$};
	\node[below] at (2,0) {$H$};
	\node[above] at (-1,0) {$\Z_2\ast \Z_2$};
	\node[above] at (1,0) {$\Z_2\ast \Z_2$};
	
	\draw[fill] (0,0) circle [radius=0.05];
	\draw[fill] (2,0) circle [radius=0.05];
	\draw[fill] (-2,0) circle [radius=0.05];
	
	\draw (-2,0)--(0,0)--(2,0);
	\end{tikzpicture}
	
	\caption{A splitting of $G_3$.}
	\label{figure: $G_3$ splitting}
\end{figure}

$H$ is virtually a surface group, therefore $X$ consists of spaces quasi-isometric to $\H^2$, which we will call $\tilde{H}$, and Euclidean planes. Applying the same reasoning as in Example \ref{example: s2 torus}, we can conclude that $\bnd X$ is path connected. This example is important for two reasons. First, the way we split the group is vital for applying Theorem \ref{combo}. If we split the group in the `obvious' way:
\[H\ast_{\Z_2\ast \Z_2} H\]
then the virtual $\Z^2$ is not captured in the splitting (i.e. this is not a peripheral splitting). Because of this, it is no longer true that if two rays go through the same sequence of $\tilde{H}$s, then the rays are asymptotic. Second, this is the type of group for which Hruska-Ruane's results on semistability do not apply \cite{HruskaRuaneSemistable}. However, it is known that RACGs  are semistable at infinity \cite{MihalikArtinCoxeterSemistable}.

\end{example}

We will refer to all three of these examples again in Section \ref{section:isolated flats} to describe the isolated flats structure on each group, and again in Section \ref{section block decomp} to describe the block decompositions of each space.
	
\section{The Visual Boundary} \label{section visual boundary}
We refer the reader to \cite{BH99} for an introduction to the theory of $\cat(0)$ spaces and groups. Throughout this section, $X$ is a proper $\cat(0)$ space. 

We describe the topology on $\overline{X}=X\cup \bnd_{x_0} X$ and then restrict this topology to the visual boundary.

\begin{defn}[$\X$ and the cone topology]
	
	Fix a basepoint $x_0\in X$. Let $\bnd_{x_0} X$ be the set of geodesic rays $c\colon [0,\infty)\to X$ based at $x_0$. Let $\overline{B}(x_0,r)$ be the closed ball of radius $r$ and $\pi_r$ be the projection map to this ball. Define sets of the following form:
	
	\[ U(c,r,D):= \{ x\in \X : d(x,x_0)>r, d(\pi_r(x),c(r))<D\}\]
	where $c$ is a geodesic segment or ray based at $x_0$, $r,D>0$. A neighborhood basis for the \textit{cone topology} on $\X$ consists of sets of the form $U(c,r,D)$ along with metric balls in $X$.
	
	The cone topology can be restricted to $\bnd_{x_0} X$ by taking sets of the following form to be a neighborhood basis:
	
	\[ U(c,r,D)= \{ c'\in \bnd_{x_0} X: d(c(r),c'(r))<D\}\]
	with $r,D>0$.
		\end{defn}
	
	Due to the follow result in \cite{BH99}, the topology on $\bnd_{x_0} X$ is independent of choice of basepoint. Therefore we can denote the visual boundary as $\bnd X$ without ambiguity. In practice, however, it often useful to fix a basepoint and work with $\bnd_{x_0} X$.
	
	\begin{prop}[{\cite[II.8]{BH99}}]
		For any choice of $x_0$ and $x_1$, $\bnd_{x_0} X$ and $\bnd_{x_1} X$ are homeomorphic.
	\end{prop}

We use a metric on $\bnd X$ which was first introduced by Osajda. We use this metric in Section \ref{section:decomp}.

\begin{defn}[Metric on $\bnd X$]
	Fix $x_0\in X$ and $D>0$. The metric $d_D$ on $\bnd_{x_0} X$ is defined as follows: given two rays $c,c'\in \bnd_{x_0} X$, $d_D(c,c')=\dfrac{1}{r}$ where $r\in [0,\infty)$ such that $d(c(r),c'(r))=D$. Since $d(c(t),c'(t))$ increases monotonically with $t$, there is a unique $r$ satisfying $d(c(r),c'(r))=D$.
	
	This metric is compatible with the cone topology \cite{OsajdaSwic}. Furthermore, the spaces $(\bnd X, d_D)$ and $(\bnd X, d_{D'})$ are isometric, so the metric is independent of choice of $D$.
\end{defn}

\section{Relatively Hyperbolic Group and Peripheral Splittings} \label{rel hyp section}

In this section, we define relatively hyperbolic groups, the Bowditch boundary, and state a theorem of Tran which relates the Bowditch boundary to the $\cat(0)$ boundary. We then introduce the maximal peripheral splitting of a relatively hyperbolic group. We will use the definition of relative hyperbolicity from \cite{Yaman}. This definition uses an action of a group on a compact, metrizable topological space $M$, which will ultimately be the Bowditch boundary.

\begin{defn}[Convergence action]
Let $M$ be a compact metrizable space and $G$ a finitely generated group. An action of $G$ on $M$ is a \textit{convergence action} if it is properly discontinuous on distinct triples of $M$.
\end{defn}

\begin{defn}[Geometrically finite]
A convergence group action of $G$ on $M$ is \textit{geometrically finite} if every point of $M$ is either a bounded parabolic point or a conical limit point. 
\end{defn}

\begin{defn}[Relatively hyperbolic]
	Let the action of $G$ on $M$ be a geometrically finite convergence action and let $\bP$ be a collection of subgroups of $G$. Then $G$ is \textit{hyperbolic relative to $\bP$} if $\bP$ is exactly the collection of maximal parabolic subgroups. 
\end{defn}

\begin{defn}[Bowditch boundary]
	The space $M$ is uniquely determined by the pair $(G,\bP)$ and is the \textit{Bowditch boundary}, denoted $\bnd(G,\bP)$.
\end{defn}

Since we do not deal with the Bowditch boundary directly, we will not define all the terms above. See \cite{Yaman} for precise definitions. We need the Bowditch boundary in order to apply Theorem \ref{Tran Theorem} and to use a handful results of Bowditch which are summarized in Theorem \ref{thm: peripheral splitting}. Both theorems can be found below.

There are many other ways to define a relatively hyperbolic group, see for example \cite{FarbBoundedCoset, GrovesManningCusped}. Each of these ultimately has $G$ acting on a $\delta$-hyperbolic metric space and takes the boundary of this space to be the Bowditch boundary. This is equivalent to our definition of the Bowditch boundary. Furthermore, all of these definitions of relatively hyperbolic are equivalent.

Suppose $G$ is a $\cat(0)$ group acting geometrically on a $\cat(0)$ space $X$ and $(G,\bP)$ is also relatively hyperbolic. We have so far introduced two different boundaries for $G$: the $\cat(0)$ boundary and the Bowditch boundary. Work of Tran tells us how to relate these two boundaries:

\begin{thm}[\cite{Tran13}] \label{Tran Theorem}
	Let $(G,\bP)$ be a relatively hyperbolic group acting geometrically on a $\delta$-hyperbolic or $\cat(0)$ space $X$. The quotient space formed from $\bnd X$ by collapsing the limit sets of each $P\in \bP$ to a point is $G$-equivariantly homeomorphic to the Bowditch boundary $\bnd (G,\bP)$.
\end{thm}

We will now introduce the maximal peripheral splitting of a relatively hyperbolic group. This plays a vital role in establishing the block decomposition of the relatively hyperbolic groups we study.

\begin{defn}[Peripheral Splitting]
	Let $(G,\bP)$ be a relatively hyperbolic group. A \textit{peripheral splitting} of $(G,\bP)$ is a splitting of $G$ as a finite bipartite graph of groups $\G$ whose vertices have two type: peripheral and component. Furthermore, the collection of subgroups of $G$ which are conjugate to a peripheral vertex group is the same as $\bP$, and $\G$ does not contain a vertex of degree 1 which is contained in the adjacent peripheral vertex group.
\end{defn}

A splitting of this form is called \textit{trivial} if one of the vertex groups is equal to $G$ and \textit{nontrivial} otherwise. A splitting $\mathcal{H}$ is a \textit{refinement} of another splitting $\G$ if $\G$ can be obtained from $\mathcal{H}$ by a sequence of foldings over edges which preserves the peripheral/component coloring. A peripheral splitting is \textit{maximal} if it is not a refinement of any other splitting.

We combine a number of results of Bowditch about peripheral splittings into one large theorem:
\begin{thm}[Peripheral Splittings Results] \label{thm: peripheral splitting}
	Let $(G,\bP)$ be a 1-ended relatively hyperbolic group. Then the following hold:
	
	\begin{enumerate}
		\item if each $P\in \bP$ is finitely presented, 1- or 2-ended, and does not contain an infinite torsion subgroup, then $\bnd (G,\bP)$ is connected \cite[10.1]{Bow12} and locally connected \cite[1.5]{Bow01},
		
		\item if $G_v$ is a component vertex group of a peripheral splitting, then $G_v$ is hyperbolic relative to $\bP_v:= \{ P\cap G_v: P\in \bP\}$ \cite[1.3]{Bow01},
		
	\item if $G_v$ is a component vertex of a maximal peripheral splitting and both $\bP$ and $\bP_v$ consist of finitely generated, 1- or 2-ended groups with no infinite torsion subgroups, the $\bnd (G_v,\bP_v)$ is connected and locally connected \cite[1.3 and 1.5]{Bow01}.

		
		
		\item if $G_v$ is a component vertex of a maximal peripheral splitting and both $\bP$ and $\bP_v$ consist of finitely generated, 1- or 2-ended groups with no infinite torsion subgroups, then $\bnd (G_v,\bP_v)$ contains no global cut points \cite[1.4]{Bow01}.
	\end{enumerate}
\end{thm}

In the setting we are dealing with ($\cat(0)$ groups with isolated flats) each component vertex is virtually $\Z^n$ for $n\geq 2$, and so 1-ended with no infinite torsion subgroups. Furthermore, if $G_v$ is a component vertex, $\bP_v$ consists of 1- or 2-ended groups.  We use the first point of this theorem in order to accurately restate a result of Hruska-Ruane in Section \ref{section:decomp}. The second point allows us to make sense of the last two. The third point is used at the end of Section \ref{section:decomp} and the final point is used in a proposition stated in Section \ref{section:local conn at flats}.

\section{Isolated Flats} \label{section:isolated flats}

Here we define $\cat(0)$ isolated flats and outline some results about their large scale geometry. 

\begin{defn}[Flat]
	A subset $F\subset X$ is a \textit{flat} if it is isometric to $\E^n$ for some $n\geq 2$.
\end{defn}

There are many equivalent definitions of isolated flats. Hruska and Kleiner prove the equivalence of these definitions in \cite{HK05}. We will take the following as our definition:

\begin{defn}[$\cat(0)$ space with isolated flats]
	Let $X$ be a $\cat(0)$ space admitting a geometric group action by $G$. Let $\F$ be a $G$-invariant collection of flats. Then $X$ has \textit{isolated flats} if the following two properties hold:
	
	\begin{enumerate}
		\item Maximality: There is  $D<\infty$ such that for any flat $F\subset X$, there is a flat $F'\in \F$ with $F\subset N_D(F')$. 
		\item Isolated: For any $r>0$, there is a $\rho=\rho(r)>0$ such that for any distinct $F,F'\in \F$,	
		$$\text{diam}(N_r(F)\cap N_r(F'))<\rho$$		
	\end{enumerate}
\end{defn}

We say a group $G$ is \textit{$\cat(0)$ with isolated flats} if it acts geometrically on a $\cat(0)$ space with isolated flats.

\begin{example}~

	\begin{enumerate}
		\item Return to Example \ref{example: s2 torus}. The universal cover of the surface group amalgam in Figure \ref{figure: s2 torus} is a $\cat(0)$ space with isolated flats. The fundamental group acts geometrically on the universal cover. The flats, coming from the universal cover of the torus, are isolated in the sense that there is a hyperbolic plane separating them.
		
		\item Consider Example \ref{example: punctured tori}. The universal cover of the two tori with boundary components with identifications given in Figure \ref{figure:tori with boundary} is a $\cat(0)$ space with isolated flats. The $\bR^2$'s coming from the Klein bottle are separated by truncated $\H^2$'s.
		
		\item From Example \ref{example: RACG}, the Davis complex of this RACG is $\cat(0)$ with isolated flats admitting a geometric group action by the RACG. Here, the flat comes from the square in the defining graph. As with the example above, pairs of flats are separated by (something quasi-isometric to) hyperbolic planes.
		
		\item Let $K$ be the figure 8-knot and $M=S^3\setminus K$ and let $X$ be the universal cover of $M$. Then $X$ is isometric to truncated $\H^3$ and is a $\cat(0)$ space with isolated flats, with the lifts of the torus-boundary of the manifold making up the flats.
	\end{enumerate}
\end{example}

One of the major theorems from Hruska and Kleiner's work is that there are a number of equivalent characterizations of $\cat(0)$ spaces with isolated flats. 

\begin{thm}[{\cite[1.2.1]{HK05}}]
	Let $X$ be a $\cat(0)$ space admitting a geometric group action by $G$. Then the following are equivalent
	
	\begin{enumerate}
		\item $X$ has isolated flats.
		\item Each component of the Tits boundary $\partial_T X$ is either an isolated point or a standard Euclidean sphere.
		\item $X$ is relatively hyperbolic with respect to a family of flats $\F$.
		\item $G$ is relatively hyperbolic with respect to a collection of virtually abelian subgroups of rank at least 2.
	\end{enumerate}
\end{thm}

One key fact we use to establish that $\cat(0)$ groups with isolated flats satisfy condition (3) in Theorem \ref{combo} is a quasiconvexity result from Hruska-Ruane. Recall the definition of quasiconvex:

\begin{defn}[Quasiconvex]
	Let $X$ be a geodesic metric space. A subset $C$ is \textit{$\kappa$-quasiconvex} if for any pair of points $x,y\in C$, any geodesic between $x$ and $y$ is in $N_\kappa(C)$. 
	
	We say a subset is \textit{quasiconvex} if it is $\kappa$-quasiconvex for some $\kappa$. 
\end{defn}

\begin{thm} [{\cite[5.6]{HR17}}] \label{Quasi-convex Constant}
	
	Let $X$ be a $\cat(0)$ space with isolated flats with respect to $\F$. There is a constant $\kappa>0$ such that the following holds:
	\begin{enumerate}
		\item Given two flats $F_1,F_2\in \F$ with $c$ the shortest length geodesic from $F_1$ to $F_2$, then $F_1\cup F_2\cup c$ is $\kappa$-quasiconvex in $X$.
		\item Given a point $p$ and a flat $F\in \F$ with $c$ the shortest path from $p$ to $F$, then $c\cup F$ is $\kappa$-quasiconvex in $X$. In particular, if $c'$ is any geodesic joining a point of $c$ to a point of $F$, then $c'$ intersects the $\kappa$-neighborhood of the endpoint in the flat of $c$.
	\end{enumerate}
\end{thm}

\section{Block Decompositions} \label{section block decomp}

This section is devoted to defining a block decomposition of a space, providing some motivating examples, and concludes with the proof of Theorem \ref{combo}. A block decomposition was first introduced in \cite{Moo10} as a generalization of the results from \cite{CK00}. In both papers, the authors focus on spaces which admit geometric group actions by groups of the form $G=A\ast_C B$. The spaces on which these groups act can be viewed as the union of closed, convex pieces, called blocks, and the intersection of these blocks follows the subgroup intersection described in the Bass-Serre tree for the splitting. While the definition of a block decomposition is independent of a group action, this paper's motivating examples of such a decomposition come from the splitting of a group.

We will use the following definition from \cite{Moo10}:

\begin{defn} \label{block decomposition}
	A \textit{block decomposition} of a CAT(0) space $X$ is a collection $\B$ of closed, convex subsets (called blocks) such that the following conditions hold:
	
	\begin{enumerate}
		\item 	$X = \bigcup \limits_{B\in \B} B$,
		\item each block intersects at least two other blocks,
		\item Parity condition: every block has a $(+)$ or $(-)$ parity such that two blocks intersect only if they have opposite parity,
		\item $\epsilon$-condition: there is an $\epsilon>0$ such that two blocks intersect iff their $\epsilon$-neighborhoods intersect.
	\end{enumerate}
	
\end{defn}

We will mostly use the facts that each $B$ is closed and convex and that $X=\bigcup\limits_{B\in \B} B$. Conditions (2)-(4) are there to prevent any degenerate cases (e.g. splitting $X$ into only two pieces, decomposing a plane as a union of lines). 

\begin{defn}
	A \textit{wall} is a non-trivial intersection of blocks and we denote the collection of walls $\W$. 
\end{defn}

Since each block is convex, note that each wall is convex as well. The nerve, denoted $\Nerve(\B)$, is a graph which records block intersections. There is a vertex for each block and an edge if two blocks (non-trivially) intersect. We will use the following notation: If $v$  and $w$ are vertices of $\Nerve(\B)$, then $B_v, B_w\in \B$ are the blocks corresponding to $v$ and $w$, respectively. Given an edge $e$ with vertices $v$ and $w$, then $B_v$ and $B_w$ intersect in a wall by definition and we denote this wall $W_e$.

Following results from \cite{CK00} and \cite{Moo10}, we get that the nerve is always a tree.
\begin{lem}
	$\Nerve(\B)$ is a tree.
\end{lem}

Since $\Nerve(\B)$ is a tree, we will denote it $\T_\B$ or $\T$ when $\B$ is understood. Here are some examples of block decompositions:

\begin{figure}[h]
	\includegraphics[width=.4\textwidth]{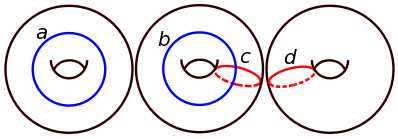}
	\centering
	\caption{Three tori with the curve $a$ identified with $b$ and the curve $c$ identified with $d$.}
	\label{fig:torus complex}
\end{figure}

\begin{example} \label{example: Croke Kleiner}

	In \cite{CK00}, the authors describe what they call a torus complex, see Figure \ref{fig:torus complex}. Let $G$ be the fundamental group of the torus complex and $X$ the universal cover. The space $X$ is $\cat(0)$ and admits a block decomposition coming from a the following splitting of $G$:
	$$G = \left( F_2\times \Z \right) \ast_{\Z^2} \left( F_2\times \Z \right)$$
	The fundamental group of two tori with a pair of curves identified is $F_2\times \Z$ which acts geometrically on $T_4\times \bR$, where $T_4$ is the valence 4 tree. The blocks come from the lifts of the middle torus and either the left or right torus, meaning each $B\in \B$ is isometric to $T_4\times \bR$. The walls are the lifts of the middle torus and are all isometric to $\bR^2$. 
	
	 Croke and Kleiner \cite{CK00} use $\B$ to show that by changing the angle of intersection between the curves $b$ and $c$ on the middle torus, $G$ can act geometrically on spaces $X_1$ and $X_2$ which have non-homeomorphic boundaries. This answered a question of Gromov. Later, in \cite{CK02}, Croke and Kleiner vastly extend their findings on this specific example to a class of groups they call `admissible.'
\end{example}

\begin{example} We return to Example \ref{example: s2 torus}. The universal cover of the surface amalgam from this example is $\cat(0)$ and has a block decomposition coming from the splitting of the fundamental group. Here, if $G_1$ is the fundamental group, then $G_1$ splits as
	$$G_1 = \pi_1(S_2)\ast_\Z \Z^2$$
	where $S_2$ is the surface with genus 2. The blocks of this decomposition come from the lifts of the two surfaces. Therefore there are two types of blocks: those isometric to $\H^2$ and those isometric to $\bR^2$. The walls, however, are all isometric to $\bR$. Unlike the torus complex, this group is $\cat(0)$ with isolated flats. 
\end{example}

\begin{example}
	Consider Example \ref{example: punctured tori} once again. Let $G_2$ be the fundamental group of this surface amalgam and $X_2$ its universal cover. Recall that this group admits the maximal peripheral splitting as seen in Figure \ref{fig: max periph splitting G2}. $X_2$ admits a block structure coming from this splitting. The vertex groups which are free act geometrically on truncated hyperbolic space and the vertex group which is the Klein bottle group acts geometrically on a Euclidean plane. These truncated hyperbolic planes and Euclidean planes in $X_2$ make up the block structure. That is, each block is isometric to either a Euclidean plane or a truncated hyperbolic plane.
\end{example}

\begin{example}
	Return now to the RACG in Example \ref{example: RACG}. Let $G_3$ be the fundamental group and $X_3$ the universal cover of its Davis complex. The block structure on $X_3$ once again comes from a splitting of $G_3$ which we saw in Figure \ref{figure: $G_3$ splitting}. The blocks of $X_3$ come from the subspaces on which the vertex groups act geometrically. The hexagon subgroup $H$ acts geometrically on a space quasi-isometric to $\H^2$ while $(\Z_2\ast \Z_2)^2$ acts geometrically on $\bR^2$. These two types of spaces form the blocks of $X_3$. 
\end{example}

\begin{example}
	More generally, if $X$ is a $\cat(0)$ space coming from the construction in the Equivariant Gluing Theorem (\cite{BH99} Theorem II.11.18), then $X$ admits a block decomposition coming from the splitting of the group $G=G_1\ast_H G_2$. There is a natural map $p\colon  X\to \T$, where $\T$ is the Bass-Serre tree for the splitting. If $v$ is a vertex of $\T$ then the blocks are each of the form  $p\inverse(N_{1/2}(v))$. We will be using this in Section \ref{section:fixing a space}.
\end{example}

With these examples in mind, we can say something about the way block boundaries intersect. We say a wall $W_e\in \W$ \textit{separates} blocks $B_v$ and $B_u$ if any path from $B_v$ to $B_u$ must pass through $W_e$. This is equivalent to $e$ being an edge on any path from $v$ to $u$ in $\T$. 

\begin{lem} \label{Block int}
	For two blocks $B_u,B_v\in \B$ corresponding to vertices $u,v\in \T$, then one of the following holds:
	
	\begin{enumerate}
		\item $\bnd B_u \cap \bnd B_v = \emptyset$
		\item The vertices $u$ and $v$ are adjacent and $\bnd B_u \cap \bnd B_v = \bnd W$ where $W=B_u \cap B_v$.
		\item The vertices $u$ and $v$ are not adjacent in $\T$ but $\bnd B_u \cap \bnd B_v$ is non-empty. Then $\bnd B_u \cap \bnd B_v \subset \bnd W$ where $W$ is \textit{any} wall between $B_u$ and $B_v$.
	\end{enumerate}
\end{lem}

\begin{proof}
	We may assume that $\bnd B_u \cap B_v \neq \emptyset$, so we are in either case (2) or case (3). If $u$ and $v$ are adjacent, then case (2) follows immediately. Suppose then that $\alpha\in \bnd B_u \cap \bnd B_v$ and that $u$ and $v$ are not adjacent in $\T$. Let $W$ be any wall separating $B_u$ and $B_v$. 
	
	Pick basepoints $x_0\in B_u$ and $x_1\in B_v$ and let $y$ be a point in $W$ along the geodesic segment from $x_0$ to $x_1$. Let $c_0$ and $c_1$ be rays representing $\alpha$ based at $x_0$ and $x_1$, respectively. For each $n\in \N$, let $y_n$ be a point in $W$ along the geodesic segment from $c_0(n)$ to $c_1(n)$. Since $W$ separates the blocks, these $y_n$ exist. Since $c_0$ and $c_1$ are asymptotic, there is a constant $K$ such that $d(c_0(t),c_1(t))\leq K$ for all $t$. Therefore $y_n$ is $K$-close to both $c_0$ and $c_1$. The limit of the geodesic segments from $y$ to $y_n$, therefore, is asymptotic to $c_0$ and $c_1$, so is a representative of $\alpha$. Since $W$ is convex, then $\alpha\in \bnd W$.
\end{proof}

\begin{remark}
	The following should be noted from the proof above: it is possible for $\alpha\in \bnd B$ but a ray representing $\alpha$ is disjoint from $B$ entirely. This happens above with the block $B_u$ and the ray $c_1$.
\end{remark}

We use the following terminology and the definition below from \cite{CK00}. A ray $c$ \textit{enters} a block $B$ if for some time $t$, $c(t)\in B$ and $c(t)$ is not in any other block (and therefore is not in a wall). Similarly, a ray $c$ \textit{exits} a block $B$ if for some $t$, $c(t)\in B$ and not in any wall and for some $t'>t$ $c(t')\notin B$. 

\begin{defn}
	Let $x_0\notin W$ for all $W\in \W$. The \textit{itinerary} of a ray $c$ based at $x_0$, denoted $\Itin_{x_0}(c)$, is the sequence of blocks $c$ enters. If $x_0$ is understood, then the itinerary will be denoted $\Itin(c)$. Abusing notation, if the basepoint is understood and $c(\infty)=\alpha\in \bnd X$, then $\Itin(\alpha)$ is used to denote the itinerary of the ray representing $\alpha$. 
	
\end{defn}

In order for this definition to be well defined, we need the following lemma:

\begin{lem}
	If $\B$ is a block decomposition for $X$, then there are points of $X$ which are not in any walls. In fact, for each $B\in \B$, there are points of $B$ which do not lie in any walls.
\end{lem}

\begin{proof}
	Let $B\in \B$ and let $W$ be a wall of $B$ such that $B\cap B_0 = W$. Define $D=\inf_{W'\subset B} d(W,W')$.
	
	If $D>0$, then all points in $N_D(W)\setminus W \cap B$ are points of $B$ which do not lie in any walls. Since $D>0$, then this set is non-empty.
	
	Assume, for contradiction, that $D=0$. Then pick $\delta< \epsilon$ and consider $N_\delta(W)$. Since $D=0$, there is some wall $W'\subset B$ which non-trivially intersects $N_\delta(W)$. Let $B\cap B'=W'$. Since $\delta<\epsilon$, $N_\delta(W)\cap W'$ means that in fact $N_\epsilon(B_0)\cap W'\neq 0$. Since $W'\subset B'$ and applying the $\epsilon$-condition on the block decomposition, $B_0\cap B'\neq \emptyset$. But then $B, B', B_0$ all pairwise intersect, contradicting the parity condition. Therefore $D>0$.
\end{proof}	

\begin{remark} It is important to note that it is possible for $\alpha\in \bnd B$ but for $B\notin \Itin(\alpha)$. This can happen as in the proof of Lemma \ref{Block int}, where $c_0$ represent $\alpha$, $\alpha\in \bnd B_v$, but $\Itin_{x_0}(c_0)=\Itin_{x_0}(\alpha)=\{B_u\}$. In other words, this can happen when $\alpha$ lies in the boundary of multiple blocks.
\end{remark}

The lemma below follows exactly in the same way as \cite{CK00} (Lemma 2), which uses the fact that each $B$ is convex and and $B$'s topological frontier is covered by the collection of blocks corresponding to the link in $\T$ of the vertex corresponding to $B$.

\begin{lem} [{\cite[2]{CK00}}] \label{Itinerary to Nerve} 
	Let $X$ be a $\cat(0)$ space with a block decomposition $\B$ and corresponding nerve $\T$. The itinerary for any $\alpha\in \bnd X$ consists of blocks which correspond to vertices on a geodesic segment or ray in $\T$.  
\end{lem}

This lemma gives rise to the following definition:

\begin{defn}

	Fix a basepoint $x_0$ which is not in any wall and let $c$ be a geodesic ray based at $x_0$ representing the point $\alpha\in \bnd X$. The point $\alpha$ is \textit{rational} if $\Itin(c)$ is finite. The point $\alpha$ is \textit{irrational} if $\Itin(c)$ is infinite.
\end{defn}

In order for this definition to be useful, being rational or irrational must be independent of representative for $\alpha$.

\begin{lem}
	If $c$ represents $\alpha$ and $c$ is rational, then all representatives of $\alpha$ are rational. 
\end{lem}

\begin{proof}
	Let $\alpha \in \bnd X$ and let $c\colon [0,\infty)\to X$ be a ray representing $\alpha$ based at $x_0$ be a finite itinerary and let $B$ be the final block of this itinerary. So $\alpha\in \bnd B$. Let $c'\colon [0,\infty)\to X$ be a ray representing $\alpha$ based at $x_1$, and suppose, for contradiction, that $\Itin_{x_1}(c')$ is infinite. By Lemma \ref{Itinerary to Nerve}, $\Itin(c')$ represents a geodesic ray in $\T$. 
	
	First, we show that if $d(v_B,v_{B'})\geq 2n$ for $v_B,v_{B'}\in \T$, then $d(B,B')\geq 2n\epsilon$ in $X$. Let $B=B_1,B_2,...,B_{2n}=B'$ denote the $2n$-blocks between $B$ and $B'$. To see $d(B,B')\geq 2\epsilon$, notice if $d(v_B,v_{B'})=2$, then by the $\epsilon$-condition, $d(B,B')\geq 2\epsilon$ since the $\epsilon$-neighborhoods of the blocks do not intersect. So let $x\in B$ and $x'\in B'$ and let $[x,x']$ be the geodesic segment between them. For each $0\leq i \leq n$, there are $z_i\in B_{2i}$ such that $d(z_i,z_{i+2})\geq 2\epsilon$. Therefore $d(x,x')\geq 2n\epsilon$. 
	
	Now, for any $K>0$, there exists some $n$ such that $2n\epsilon>K$. Furthermore, since $\Itin(c')$ represent a geodesic ray in $\T$, there is a block $B'$ such that $d(v_B,v_{B'})\geq 2n$. Thus $c'$ is not in the $K$-neighborhood of $B$ for any $K$. But $c$ and $c'$ are asymptotic, a contradiction. Thus $\Itin(c')$ must be finite.
	
\end{proof}

\begin{corollary}
	If $\Itin(c)$ is infinite for some ray $c$, then all rays representing $c(\infty)$ have infinite itineraries.
\end{corollary}

\begin{proof}
	Suppose $c'$ represents $c(\infty)$ and is finite. Then by the argument above, all rays representing $c'(\infty)$ have finite itineraries, which cannot happen since $\Itin(c)$ is infinite.
\end{proof}
Let $\partial_R X$ and $\partial_I X$ denote the rational and irrational boundary points of $\bnd X$, respectively. Then we have:
\begin{enumerate}
	\item $\partial_R X = \bigcup\limits_{B\in \B} \bnd B$
	\item $\partial_I X = \bnd X \setminus \partial_R X$
\end{enumerate}

Once a basepoint is fixed we can use Lemma \ref{Itinerary to Nerve} to create a map \linebreak $\eta\colon  \partial_I X \to \bnd \T$, mapping a point $\alpha$ to the end of $\T$ representing its itinerary at that basepoint. This map is surjective but not necessarily injective. The points in $\eta \inverse(\alpha)$ are the points of the $\bnd X$ with the same itinerary as $\alpha$. We will care when $\alpha$ is the only boundary point with its itinerary.

\begin{defn}
	If $\alpha$ is irrational and is the only point in $\bnd X$ with its itinerary, then we say $\alpha$ is \textit{lonely}. We will denote the set of lonely rays $\lonely X$.
\end{defn}

\begin{thm}\label{combo}
	Let $X$ be a $\cat(0)$ space and $\B$ a block decomposition. Suppose the following conditions holds:
	
	\begin{enumerate}
		\item for each $B\in \B$, $\bnd B$ is path connected,
		\item for each $W\in \W$, $\bnd W$ is non-empty,
		\item all irrational rays are lonely
	\end{enumerate}
	then $\bnd X$ is path connected.
\end{thm}

\begin{remark}
	It should be noted that we only assume $\bnd W$ is non-empty. We do not need that $\bnd W$ is path connected. The walls having non-empty boundary is equivalent to $\bnd X$ being connected.
\end{remark}

Going forward, we shall refer to condition (3) as the \textit{lonely condition}.

Before proving Theorem \ref{combo}, it should be noted that conditions (1) and (2), without (3), are insufficient for guaranteeing $\bnd X$ being path connected. Returning to Example \ref{example: Croke Kleiner}, the block decomposition satisfies conditions (1) and (2), but $\bnd X$ is not path connected. In \cite{CMT}, the authors explicitly construct a pair of distinct irrational rays with the same itinerary which are in a different path component from the union of the block boundaries. In that example, there is something akin to the Topologist's Sine Curve in the boundary, which obstructs path connectivity. Theorem \ref{combo} shows that the lonely condition is sufficient for avoiding this pathology.

We will prove Theorem \ref{combo} in two parts. First we will show that $\partial_R X$ is path connected. We then show if $\partial_R X$ is path connected, then $\partial_R X \cup \lonely X$ is path connected as well. By the lonely condition, $\lonely X = \bnd_I X$, concluding the proof.

\begin{proposition}
	Suppose for each $B\in \B$, $\bnd B$ is path connected and for each $W\in \W$, $\bnd W$ is non-empty. Then $\partial_R X$ is path connected.
\end{proposition}

\begin{proof}
	Fix $x_0\in X\setminus \W$.	Let $x_0\in B$ and $B'$ be an adjacent block. We show that $\bnd B \cup \bnd B'$ is path connected then proceed by induction. Since block boundaries are path connected, it suffices to show that there is a path from a point in $\bnd B$ to a point in $\bnd B'$. Let $\alpha\in \bnd B$ and $\alpha' \in \bnd B'$. Since $B$ and $B'$ are adjacent, there is a $W\in \W$ such that $B\cap B' = W$. Let $\beta\in \bnd W$, which exists since $\bnd W$ is non-empty. Since $\bnd W = \bnd B \cap \bnd B'$ and each block has path connected boundary, there are paths from $\alpha$ to $\beta$ and from $\beta$ to $\alpha'$ in $\bnd B \cup \bnd B'$. The concatenation of these paths is a path from $\alpha$ to $\alpha'$.
	
	Let $\Itin(\alpha')= \{B_1, B_2, ..., B_n\}$ with $B_1=B$. The union of the boundary of adjacent blocks is path connected, so $\bnd B_{n-1} \cup \bnd B_n$ is path connected. By the induction hypothesis, $\bnd B_1 \cup \bnd B_2 \cup ... \cup \bnd B_{n-1}$ is path connected, and therefore $\bnd B_1 \cup \bnd B_2 \cup ... \cup \bnd B_n$ is path connected, which concludes the proof. 
\end{proof}

Before proving the next proposition, we need the following lemma from \cite{CK02}:

\begin{lem} \label{Croke Kleiner Converging Rays}
	Let $X$ be a locally compact CAT(0) space with a closed, convex subset $Y$. Let $x_0\in X$ and $\alpha_i\in \bnd X$ with $\alpha_i\to \alpha$ and the rays $r_i$ based at $x_0$ representing $\alpha_i$ and $r$ the ray based at $x_0$ representing $\alpha$. If $r_i$ intersects $Y$ for all $i$, then either $r$ intersects $Y$ or $\alpha \in \bnd Y$.
\end{lem}

\begin{proposition} \label{lonely plus rational is pconn}
	If $\partial_R X$ is path connected, then $\partial_R X \cup \lonely X$ is path connected.
\end{proposition}

\begin{proof}
	Fix $x_0\in X\setminus \W$ and let $\alpha \in \lonely X$ with $\Itin(\alpha) = \{B_1, B_2, ...\}$. Let $c\colon [0,\infty)$ be a path in $\partial_R X$ such that $c([n-1,n])\subset B_n$ for all $n\in \N$. Our goal is to extend $c$ to $[0,\infty]$ such that $c(\infty)=\alpha$. 
	
	Let $\{t_n\}\subset [0,\infty)$ be a monotone sequence tending to infinity.

	Since $\bnd X$ is compact then, by passing to a subsequence if necessary, $\{c(t_n)\}$ converges to some $\alpha'$. It suffices to show that $\alpha = \alpha '$, which we will do by showing $\Itin(\alpha) = \Itin(\alpha')$.
	
	For each $n$, $\Itin(c(t_n))=\{B_1, B_2, ..., B_{k_n}\}$ where $k_n = \lceil t_n \rceil$. Let $r_n$ be the ray based at $x_0$ representing $c(t_n)$ and let $r$ be the ray based at $x_0$ representing $\alpha'$. By our construction, $r_n\to r$.
	
	Fix a block $B_m$ in $\Itin(\alpha)$. There is an $N$ such that for all $n>N$, $B_m\in \Itin(c(t_n))$ and so $r_n\cap B_m \neq \emptyset$. By Lemma \ref{Croke Kleiner Converging Rays}, either $r\cap B_m\neq \emptyset$ or $r\in \bnd B_m$. If the former, then $B_m\in \Itin(\alpha')$. If the latter is true, pick points $x_n(m)\in r_n \cap B_m$ such that $d(x_0, x_n(m))\to \infty$. Such a sequence exists since $r\in \bnd B_m$ and $r_n\to r$. Then, after possibly passing to a subsequence, the geodesic segments from $x_0$ to $x_n(m)$ converge to a ray $\rho_m$ with $\rho_m(\infty) = \alpha'$. For each $i<m$, $B_i$ separates $B_m$ from $x_0$, so $B_i\in \Itin(\alpha')$.
	
	But if this is true for some $m$, then it is true for all $M>m$ and so $\alpha \in \bnd B_M$ for all $M>m$. Consider the ray $\rho_M$ by the same construction as the previous paragraph. Using the same argument, $B_i\in \Itin(\rho_M(\infty)) = \Itin(\alpha')$ for all $i<M$. Since $m<M$, this means $B_m\in \Itin(\alpha')$. Therefore $\alpha'$ and $\alpha$ have the same itinerary. Since $\alpha$ is lonely, $\alpha = \alpha'$.
\end{proof}

\begin{remark}
	Proving path connectivity of the boundary does not require each block to have a path connected boundary. The proposition only assumes that $\partial_R X$ is path connected. This distinction will come up later when we show all 1-ended $\cat(0)$ groups with isolated flats have path connected visual boundaries.
\end{remark}

\begin{remark}
	While the lonely condition is sufficient, it is not necessary for path connectivity. This comes up for the following group
	\[G=F_2\times \Z = \Z^2\ast_\Z \Z^2\]
	$G$ acts the universal cover of two tori with a curve on each identified. The block decomposition comes from viewing $G$ as the amalgamated product of two $\Z^2$s over a $\Z$. The boundary is path connected since it is a suspension of a Cantor set, but irrational rays are not lonely.
\end{remark}

\section{Fixing a Space}\label{section:fixing a space}

Normally when dealing with $\cat(0)$ groups and their boundaries, one needs to be particularly careful about the space they are working with. This is because Croke and Kleiner \cite{CK00} showed  that, unlike in the hyperbolic setting, $\cat(0)$ groups do not have a well defined visual boundary. When it comes to $\cat(0)$ groups with isolated flats, however, work of Hruska and Kleiner \cite{HK05} shows that the visual boundary of these groups is well defined . In this section, we will fix the space on which a 1-ended $\cat(0)$ group with isolated flats will act on and define the block decomposition for this space.

Throughout this section, let $G$ be a $\cat(0)$ group which is hyperbolic relative to $\bP$, the collection of flat-stabilizers so that $(G,\bP)$ is the relatively hyperbolic pair. Let $G$ act geometrically on some $\cat(0)$ space $X$. 

We want more control over the space on which $G$ acts, so we use the peripheral splitting theorem of Bowditch introduced in Section \ref{rel hyp section}, a convex splitting theorem of Hruska and Ruane \cite{HR17}, and a combination theorem of Bridson and Haefliger \cite{BH99} to fix a different space which admits a geometric group action by $G$. We quote the two remaining theorems below. But first, a definition.

\begin{defn}
	Let $G$ be a CAT(0) group acting geometrically on a $\cat(0)$ space $X$. A subgroup $H\leq G$ is \textit{convex} if $H$ stabilizes a closed, convex subspace $Y$ of $X$, and $H$ acts cocompactly on $Y$. Therefore $H$ act geometrically on $Y$, a $\cat(0)$ space.
\end{defn}

\begin{thm}[Convex Splitting Theorem, {\cite[1.3]{HR17}}]
	Let $G$ act geometrically on a $\cat(0)$ space $X$. Suppose $G$ splits as the fundamental group of a graph of groups $\G$ such that each edge group of $\G$ is convex. Then each vertex group is convex as well. In particular, each vertex group is a $\cat(0)$ group.
\end{thm}

Let $\G$ be the maximal peripheral splitting for our group $G$. Since $G$ has isolated flats, all peripheral subgroups in the graph of groups $\G$ are virtually abelian and therefore so are all edge groups. By the Flat Torus Theorem \cite{BH99}, all virtually abelian subgroup of a $\cat(0)$ group are convex and therefore we can apply the Convex Splitting Theorem to the maximal peripheral splitting of $G$.

Let $\C$ be the collection of component vertices and $\P$ the set of peripheral vertices of $\G$. For each $v\in \C$ with vertex group $G_v$, let $X_v$ be the convex subspace of $X$ on which $G_v$ acts geometrically. For $w\in \P$, the vertex group $G_w$ is virtually abelian and therefore acts geometrically on some Euclidean space $F_w$. 

Each edge $e$ of $\G$ is incident to a peripheral vertex $w$ and a component vertex $v$. Let $F_e$ be the flat subspace of $F_w$ on which $P_e$ acts geometrically. Each edge also comes equipped with a monomorphism $\phi_e\colon P_e\to G_v$. The Flat Torus Theorem gives a $\phi_e$-equivariant isometric embedding of $F_e$ into $X_v$. 

With all of this set up, we can use the following theorem to build our desired $X$:

\begin{thm}[Equivariant Gluing Theorem, {\cite[II.11.18]{BH99}}]\label{thm:equivariant gluing}
	Let $\Gamma_0$, $\Gamma_1$, and $H$ be groups acting geometrically on complete CAT(0) spaces $X_0, X_1$, and $Y$, respectively. Suppose for $j=0,1$, there exists monomorphisms $\phi_j\colon H \hookrightarrow X_j$ and $\phi_j$-equivariant isometric embeddings $f_j\colon  Y\to X_j$, then $\Gamma=\Gamma_0\ast_H \Gamma_1$ acts geometrically on a CAT(0) space $X$.
\end{thm}

The space uses the Bass-Serre tree $\T$ for the splitting $\G$ of $G$. For each component vertex $v$ of $\T$, take a copy of $X_v$ and for each peripheral vertex $w$ of $\T$, take a copy of $P_w$. We shall refer to these spaces as \textit{vertex spaces}. Then for each edge $e$ incident to $v$ and $w$, take a copy of $F_e\times [0,1]$, which we shall refer to as \textit{edge spaces}. Identify the $F_e\times \{0\}$ part with the image of $F_e$ in $X_v$ and $F_e\times \{1\}$ is identified with the image of $F_e$ in $F_w$. Following the Bass-Serre tree, all the copies of the different vertex and edge spaces are glued up to form $X$. 

With this $X$ fixed, we can also describe the block decomposition of $X$. For each vertex space $X_v$ or $P_w$, take the closed $\frac{1}{2}$-tubular neighborhood to be a block $B$. Then each block intersects in $F_e\times \{\frac{1}{2}\}$. This defines a block decomposition of $X$.

\begin{remark}
	Throughout the remaining sections, when referring to a $\cat(0)$ group with isolated flats acting geometrically on $X$, we shall be referring to the $X$ from this construction. 
\end{remark}

Our goal for the next few sections is to show that $\bnd X$ is path connected. In Lemma \ref{IFP blocks are lonely}, we will show that with this block decomposition, all irrational rays are lonely. Therefore we just need to show that $\partial_R X$ is path connected.

In general, each component vertex is not 1-ended. When this happens, $\bnd X_v$ is cannot be path connected and therefore showing $\partial_R X$ is path connected requires more work. Therefore we introduce $Y_v$:

\begin{defn}\label{defn Y}
	For each component vertex $v$, let $Star(v)$ be the star of $v$ in $\T$. There is a continuous map $p\colon X\to \T$ which maps vertex spaces to vertices and edge spaces to edges in the obvious way. Let $Y_v= p \inverse ( Star(v))$.
	
	 Put another way, let $\mathcal{E}_v$ be the collection of edges incident to $v$ and $\P_v$ be the collection of vertices of $\T$ adjacent to $v$. Since $\T$ is the Bass-Serre tree for the maximal peripheral splitting, each $w\in \P_v$ is a peripheral vertex. Then $Y_v$ is
	
	\[Y_v := X_v \cup \left(\bigcup_{e\in \mathcal{E}_v} F_e\times[0,1]\right) \cup \left(\bigcup_{w\in \P_w} F_w \right) \]
	Put a third way, $Y_v$ consists of $X_v$ and every flat of $X$ which shares a boundary point with $Y_v$.
	
	Ultimately, we want to show that $\bnd Y_v$ is path connected for each $v\in \C$. When dealing with a generic $Y_v$, we will drop the subscript and denote it $Y$.
\end{defn}

The motivating example for needing to this is Example \ref{example: punctured tori}. The maximal peripheral splitting of this group is in Figure \ref{fig: max periph splitting G2}. Since the component vertices are free groups, they have Cantor sets for boundaries. Additionally, there is no obvious intermediate peripheral splitting which has component vertices with path connected visual boundaries. Therefore, one needs to look at one of the Cantor sets along with all the circles attached to it. We will show this space is path connected.

\section{Decomposition Spaces}\label{section:decomp}
In this section we will repeat some results from \cite{HR17} about decompositions of spaces and their relationship to CAT(0) groups with isolated flats. In particular, we establish that the map $\pi\colon \bnd X \to \bnd (G,\bP)$ from Theorem \ref{Tran Theorem} can be viewed through the lens of decomposition theory. We then use this to establish a decomposition of $\bnd Y$ and ultimately conclude that boundary points in $Y$ which are not in the boundary of a flat are locally connected. 

Much of the results here follow from elementary decomposition theory, which can be found in \cite{Dav86}. For a topological space $M$, a \textit{decomposition}, $\D$, is a partition of $M$. The decomposition map $\pi\colon M\to M/\D$ is the quotient map where each $d\in \D$ is collapsed to a point and the resulting space is given the quotient topology. In this way, decompositions of $M$ are equivalent to quotients of $M$. Any decomposition can be made into a quotient map and any quotient map can be made into a decomposition by taking point pre-images as elements of the decomposition.

\begin{defn}[Upper semicontinuous decomposition]
	A decomposition $\D$ of $M$ is \textit{upper semicontinuous} if each $d\in \D$ is compact and for each $d\in \D$ and each open $U\subset M$ containing $d$, there is an open $V\subset M$ containing $d$ such that every $d\in \D$ which intersects $V$ is contained in $U$. A quotient is \textit{upper semicontinuous} if the associated decomposition is upper semicontinuous.
\end{defn}

\begin{prop}[{\cite[I.1.1]{Dav86}}]
	Let $\D$ be a decomposition of a space $M$ with each $d\in \D$ compact. The following are equivalent:
	\begin{enumerate}
		\item $\D$ is upper semicontinuous.
		\item For an open set $U\subset M$, let $U^*$ be the union of $d\in \D$ such that $d\subset U$. Then $U^*$ is open.
		\item The decomposition map $\pi\colon  M\to M/\D$ is closed.
	\end{enumerate}
	
\end{prop}
\begin{prop}[{\cite[I.2.1 and I.2.2]{Dav86}}]
	If $\D$ is an upper semicontinuous decomposition of a Hausdorff space, then $M/\D$ is Hausdorff. If $\D$ is an upper semicontinuous decomposition of a metric space, then $M/\D$ is metrizable.
\end{prop}

Boundaries of proper CAT(0) spaces are metrizable, and therefore so are their quotients coming from upper semicontinuous decompositions.

\begin{prop}[{\cite[I.3.1]{Dav86}}]
	If $\D$ is an upper semicontinuous decomposition of $M$, then $\pi\colon M\to M/\D$ is a proper map.
\end{prop}

\begin{defn}[Monotone decomposition]
	A decomposition $\D$ of $M$ is \textit{monotone} if each $d\in \D$ is compact and connected.
\end{defn}

\begin{prop} [{\cite[I.4.1]{Dav86}}] \label{Montone if and only if connected}
	Let $\D$ be an upper semicontinuous decomposition of a space $M$. Then $\D$ is monotone if and only if $\pi \inverse(C)$ is connected whenever $C$ is a connected subset of $M/\D$.
\end{prop}

\begin{prop}[{\cite[8.5]{HR17}}] \label{Singleton points are locally connected}
	Let $\D$ be an upper semicontinuous monotone decomposition of $M$. Let $\{x\}$ be a singleton member of $\D$. If $M/\D$ is locally connected at $\pi(x)$, then $M$ is locally connected at $x$.
\end{prop}
The following definition allows us to construct subspace decompositions:

\begin{defn}[Subspace Decomposition]
	Let $\D$ be a decomposition of a Hausdorff space $M$. Let $W\subset M$ such that if $d \cap W$ is non-empty for any $d\in \D$, then $d\subset W$. The \textit{induced subspace decomposition} of $W$ is the decomposition consisting of all members of $\D$ which are contained in $W$. If $\D$ is upper semi-continuous, then the induced subspace decomposition is as well.
\end{defn}
This definition will become vital for understanding the decomposition of $\bnd Y\subset \bnd X$.

\begin{defn}[Null Family]
	A collection of subsets $\A$ of a metric space is a \textit{null family} if for each $\varepsilon>0$, only finitely many $A\in \A$ have diameter greater than $\varepsilon$.
\end{defn}

\begin{prop}[Saturation condition, {\cite[8.9]{HR17}}] \label{Null Family Saturation}
	Let $\A$ be a null family of compact sets in a metric space $M$. Suppose $q\in M$ and $q$ is not in any of the members of $\A$. Then each neighborhood $U$ of $q$ contains a smaller neighborhood $V$ of $q$ such that for each $A\in \A$, if $A\cap V \neq \emptyset$, then $A\subset U$.
\end{prop}

\begin{remark}\label{remark: null family condition}
		In view of how the metric works on $\bnd X$, the notion of a null family can be reformulated topologically as follows: Fix $x_0$ as the basepoint for the cone topology. A collection $\A$ of subspaces is a null family in $\bnd X$ if there is some $D>0$ such that for each $r<\infty$, only finitely many members of $\A$ are not contained in a set of the form $U(\cdot, r, D)$. Under the metric $d_D$ on $\partial_{x_0} X$, only those members of $\A$ which lie in $U(\cdot, r, D)$ have diameter at most $1/r$.
	
	A similar condition can be put on the cone topology of $\overline{X}$. We do not need to deal with this directly, however. Work of Bestvina says that $\overline{X}$ is metrizable and this metric is compatible with the cone topology, which is sufficient for our needs \cite{Bes96}.
\end{remark}

\begin{proposition} [{\cite[8.12]{HR17}}] \label{Family of Spheres}
	Let $X$ be a CAT(0) space with isolated flats with respect to the family $\bP$. Let $\A$ be the collection of spheres $\{ \bnd F: F\in \bP \}$. Then $\A$ is a null family of disjoint compact subsets of $\bnd X$. 
\end{proposition}

\begin{prop}[{\cite[I.2.3]{Dav86}}]
	Let $\D$ be a decomposition of a metric space $M$ such that each $d\in \D$ is compact. Let $\A$ be the collection of $d\in \D$ such that $d$ is not a singleton. Then if $\A$ is a null family, $\D$ is an upper semicontinuous decomposition. 
\end{prop}
\begin{corollary}[{\cite[8.13]{HR17}}]
	Let $G$ be a CAT(0) group with isolated flats relative to $\bP$ acting geometrically on a CAT(0) space $X$. Let $\A$ be the collection of spheres which bound maximal dimensional flats in $\bnd X$. Then the quotient map $\bnd X \to \bnd X/ \A \cong \partial (G,\bP)$ given in Theorem \ref{Tran Theorem} is upper semicontinuous and monotone.
\end{corollary}

\begin{corollary}[{\cite[8.14]{HR17}}]   \label{Non-flat points are l.c.}
	Let $G$ be 1-ended and CAT(0) with isolated flats relative to $\bP$. Then $\bnd X$ is locally connected at each $x$ which is not in the boundary of a flat. 
\end{corollary}

\begin{proof}
	From Theorem \ref{thm: peripheral splitting}, $\partial(G,\bP)$ is locally connected. Since the decomposition is upper semicontinuous and monotone, the singletons in the decomposition are locally connected as well. In this decomposition, the singletons are exactly the points which are not in the boundary of any flat.
\end{proof}

Recall the definition of $Y$(Definition \ref{defn Y}) at the end of Section \ref{section:fixing a space}: Let $p\colon X\to \T$ be the projection map onto the Bass-Serre tree $\T$. Then for a component vertex $v\in \T$, $Y_v=p\inverse(Star(v))$. Since we will being working with generic $Y_v$, we shall just denote it as $Y$.
\begin{thm} \label{monotone usc decomp of Y}
	The restriction of $\pi\colon \bnd X \to \bnd (G,\bP)$ to $\bnd Y$ defines a monotone upper semicontinuous decomposition and the image of $\bnd Y$ is $\partial (G_v, \bP_v)$, the Bowditch boundary for the corresponding component vertex. Here \linebreak $\bP_v : = \{ P\cap G_v: P\in \bP\}$.
\end{thm}

\begin{proof}
	First, see that restricting $\pi$ to $\bnd Y$ is a subspace decomposition for if $d\cap \bnd Y$, then $d\subset \bnd Y$. This is clear since for each $d\in \D$, either $d$ is a singleton or $d$ is the boundary of a maximal dimensional flat and in either of those cases, $d\cap \bnd Y$ means $d\subset \bnd Y$. Additionally this subspace decomposition is monotone. Therefore we have an upper semicontinuous decomposition $\pi\colon  \bnd Y \to \pi(\bnd Y)$. 
	
	Next, consider the relatively hyperbolic pair $(G_v,\bP_v)$. By the arguments above, the map $\pi'\colon  \bnd X_v \to \bnd (G_v, \bP_v)$ is an upper semicontinuous decomposition. Since the peripheral structure of $G_v$ is inherited from $G$, then $\pi(\bnd X_v) = \pi'(\bnd X_v)$.
	
	Lastly, we need to show that $\pi(\bnd Y) = \pi (X_v)$. To see this, note that if $d\subset \bnd Y$, then either $d$ is a singleton or $d$ is the boundary of a flat. In the former case, then $d\in \bnd X_v$. In the latter case, $d\cap \bnd X_v \neq \emptyset$. The map $\pi$ collapses all points of $d$ to a single point, therefore $\pi(d) = \pi (d\cap \bnd X_v)$, and therefore $\pi(\bnd Y) = \bnd (G_v,\bP_v)$.

\end{proof}

\begin{corollary}\label{Y non-flat points lconn}
	If $x\in \bnd Y$ and $x\notin \bnd F$ for any flat $F$, then $\bnd Y$ is locally connected at $x$.
\end{corollary}

\begin{proof}
	This follows from the fact that we have a monotone, upper semicontinuous decomposition to a locally connected space (Theorem \ref{thm: peripheral splitting}) and that $x$ is a singleton in this decomposition.
\end{proof}

\section{Proper and Cocompact Action} \label{section:proper and cocompact}

Our goal of this section is to show that the stabilizer of a flat of $X_v$ also acts properly and cocompactly on $\bnd Y$ minus the higher-dimensional flat. Recall $X_v$ is the vertex-space for a component vertex $v$ of the maximal peripheral splitting of $G$. Since we need to distinguish between flats of $X_v$ and flats of $Y$, we need to establish some notation. Throughout this section, fix a flat $F\subset X_v$ and let $F^*$ be the maximal dimensional flat of $X$ which contains $F$. By our construction of $Y$, $F^*\subset Y$ but $F^*$ is not necessarily contained in $X_v$. Let $P$ and $P^*$ stabilize $F$ and $F^*$ respectively. Lastly, define $\Omega:= \bnd Y \setminus \bnd F^*$.

The ideas from this section come from \cite{Ha17}, which proves the results in the case where $\bnd X$ is locally connected. In that setting, $Y$ and $X_v$ are the same and $F=F^*$.

The stabilizer $P$ acts properly and cocompactly on $F$. With some work, one can show that $P$ also acts properly and cocompactly on $\Omega$. Furthermore, the fundamental domains for the action of $P$ on both can be picked in a way which allows us to relate the two spaces. The following lemma of Bowditch  establishes the proper and cocompact action of $P$ on $\Omega$ \cite{Bow12}:

\begin{lem}\label{Bow1}
	Let $P$ be a group acting on topological spaces $C$ and $D$. Define the action of $P$ on $C\times D$ as the diagonal action and let $\R \subset C\times D$. If $\R$ is $P$-invariant and the projection maps $\pi_C$ and $\pi_D$ from $\R$ to $C$ and $D$ are both proper and surjective, then the following are equivalent:
	\begin{enumerate}[(a)]
		\item $P$ acts properly and cocompactly on $C$
		\item $P$ acts properly and cocompactly on $D$
		\item $P$ acts properly and cocompactly on $\R$
	\end{enumerate}
\end{lem}

Define $\bot(F)$ to be the set of geodesic rays orthogonal to $F$. Recall a geodesic ray $r\colon [0,\infty) \to X$ is \textit{orthogonal} to a convex set $C\subset X$ if $r(0)\in C$ and for every $t>0$ and any $y\in C$, the Alexandrov angle, $\angle_{r(0)}(r(t),y)$, is greater than or equal to $\pi/2$. 

We will apply this lemma where $F$ and $\Omega$ are $C$ and $D$ respectively.  Let $\R:= \{ (x,q): \text{there exists } q\in \bot(F) \text{ with } d(x,q(0))\leq A\}$, where $A$ is the diameter of the fundamental domain for the action of $P$ on $F$. Throughout, we will assume our basepoint for the topology on $\bnd Y$ is in $F$. Since $P$ acts properly and cocompactly on $F$, we need to show the $P$-invariance of $\R$ and the projection maps to $F$ and $\Omega$ are both proper and surjective.

\begin{lem}
	$\R$ is $P$-invariant.
\end{lem}

\begin{proof}
	$P$ acts by isometries on $Y$, so if $(x,q)\in \R$ and $p\in P$, then \linebreak $d(p\cdot x, p\cdot q(0))\leq A$. It is clear that $p\cdot q$ is a ray, but we need to show that it is orthogonal to $F$.
	
	For all $t$, $\pi_F(q(t))=q(0)$, so $q(0)$ is the closest point of $F$ to each point along the ray. Assume, for contradiction, that $\pi_F( p\cdot q(t)) = y$ and \linebreak $y\neq p\cdot q(0)$. Then	
	$$d(p\cdot q(t), y) < d(p\cdot q(t), p\cdot q(0))$$
	This is a strict inequality because of the uniquness of projection points. Applying $p\inverse$ to all the points gives us:
	$$d(q(t),p\inverse \cdot  y) < d(q(t),q(0))$$
	Since $p\inverse\in P$ and $P$ fixes $F$, then $p\inverse y \in F$. But $q(0)$ is the closest point of $F$ to $q(t)$, so $y=p\cdot q(0)$ and therefore $p\cdot q\in \bot(F)$.
\end{proof}

\begin{lem}\label{lemma: orthogonal rays for all bnd points}
	Let $\alpha\in \Omega$. If $r$ is a ray representing $\alpha$, then there is a ray $q\in \bot(F)$ asymptotic to $r$.
\end{lem}

To prove this lemma, we make use of the following lemma from \cite{Ha17}, which follows closely to Lemma I.5.31 in \cite{BH99}.

\begin{lem} \label{isometric embeddings converge}
	If $(Y,\rho)$ is a separable metric space, $(X,d)$ a proper metric space, $y_0\in Y$, and $K$ a compact subset of $X$, then any sequence of isometric embeddings, $c_n\colon Y \to X$, with $c_n(y_0)\in K$ has a subsequence which converges point-wise to an isometric embedding $c\colon Y\to X$.
\end{lem}

Throughout the remaining of the section, let $[x,y]$ denote the geodesic segment connecting points $x,y\in X$. 

\begin{proof}[Proof of Lemma \ref{lemma: orthogonal rays for all bnd points}]
	Fix a basepoint $x_0$, let  $x_n:=r(n)$ and let $y_n:= \pi_F(x_n)$ for all $n\in \N$. First, assume for contradiction that $y_n$ is unbounded. Then $y_n\to \eta \in \bnd F$. By \cite{HK09}, there is a constant $M>0$ such that for all $n$ $d(x_0,[x_n,y_n])<M$, where $[x_n,y_n]$ is the geodesic segment in $Y$ from $x_n$ to $y_n$. Let $m$ be the point of this geodesic closest to $x_0$. Since the geodesic $[x_n,y_n]$ is orthogonal to $F$ and $x_0,y_n\in F$, then the largest side of the triangle with vertices $x_0,y_n,m$ is the edge $[x_0,m]$. Therefore $d(x_0,y_n)<M$, contradicting $y_n$ tending to infinity.
	
	For each $n\in \N$, let $f_n\colon [0,n]\to Y$ be the geodesic $[y_n,x_n]$, which is orthogonal to $F$. For each $k$ and all $n\geq k$, consider the sequence $(f_n|_{[0,k]})$. By Lemma \ref{isometric embeddings converge}, these converge to a map $q_k\colon [0,k]\to Y$. Notice for each $\ell < k$, $q_\ell = q_k|_{[0,\ell]}$, since the tails of the sequences defining $q_\ell$ and $q_k$ are identical. Therefore, $q\colon [0,\infty)\to Y$ is a geodesic ray which is the limit of the sequence $(q_n)$. 
	
	It suffices to show that $q\in \bot(F)$ and $q$ is asymptotic to $r$. For the former, fix $y\in F$ and notice that since $[y_n,x_n]$ is a geodesic segment orthogonal to $F$,  $\angle_{y_n}(y,x_n) \geq \pi/2$ for all $n\in \N$. Then, applying Proposition II.3.3(1) in \cite{BH99}, see that the map $(x,y,p)\to \angle_p(x,y)$ is upper semicontinuous. Therefore the limiting segments $q_n$ are orthogonal to $F$ so $q\in \bot(F)$. To see that $q$ and $r$ are asymptotic, notice that $[y_n,x_n]\subset B_M(r)$, and so $q$ is $M$-close to $r$, thus $q$ and $r$ are asymptotic.

\end{proof}

\begin{corollary}
	The projection maps $\pi_F$ and $\pi_\Omega$ are surjective. 
\end{corollary}

\begin{proof}
	The lemma above shows that $\pi_\Omega$ is surjective. Using the $P$-action and the fact that $A$ is the diameter of the fundamental domain, then $\pi_F$ is surjective as well.
\end{proof}

Part of showing the projection maps are proper requires an understanding of what happens to the boundary points when a sequence of rays in $\bot(F)$ converge.

\begin{lem}
	If $(r_n)$ is a sequence of rays orthogonal to $F$ which converge to an element $r\in \bot(F)$, then $r_n(\infty)$ converges to $r(\infty)$ in the cone topology.
\end{lem}

\begin{proof}
	Fix $x_0\in F$ a basepoint for the cone topology on $\bnd Y$. For each $r_n$, let $c_n$ be the ray based at $x_0$ asymptotic to $r_n$. Since $r_n\to r$, then there is a constant $D$ such that $d(r_n(t),c_n(t))\leq D$ for all $t$. Applying Lemma \ref{isometric embeddings converge}, the rays $c_n$ converge to a ray $c$ which is asymptotic to $r$. 
	
	Fix $\epsilon>0$ and let $s>0$. $U(c,s,\epsilon)$ is a basic neighborhood of $c$ in $\bnd Y$. Since $c_n\to c$ pointwise, then there is an $N\in \N$ such that $d(c_m(s),c(s))<\epsilon$ for all $m>N$ and therefore $c_n(\infty)\to c(\infty)$.
\end{proof}

What remains is to show that the projection maps are both proper. Recall from Theorem \ref{Quasi-convex Constant} there is a universal constant $\kappa>0$ which makes certain subsets of a $\cat(0)$ space with isolated flats $\kappa$-quasiconvex. Haulmark extends this result to include geodesic rays which are orthogonal to some flat $F$.

\begin{lem}[{\cite[3.7]{Ha17}}] \label{Haulmark F-q quasiconvex}
	Let $F\in \F$ and $q\in\bot(F)$, then there exists a constant $\kappa$ such that $q\cup F$ is $\kappa$-quasiconvex in $X$.
\end{lem}

\begin{remark}
	The constants $\kappa$ in Theorem \ref{Quasi-convex Constant} and Lemma \ref{Haulmark F-q quasiconvex} above are the same $\kappa$.
\end{remark}

\begin{corollary}
	Let $F^*$ a flat of $X$, $F$ a flat subset of $F^*$ and $q\in \bot(F)$, then $q\cup F$ is $\kappa$-quasiconvex in $Y$.
\end{corollary}

\begin{proof}
	In $Y$, $\bot(F)=\bot(F^*)$ since $F$ separates $F^*$ from the rest of $Y$. $Y$ inherits the metric from $X$, so $q\cup F^*$ is quasiconvex in $Y$. Any geodesic in $Y$ from points of $F$ to points of $q$ are $\kappa$-close to $F^*$ or $q$. If they are $\kappa$-close to $F^*$, then they are $\kappa$-close to $F$ because $F$ separates $F^*$ from $Y$. Lastly, since $Y$ is convex, all geodesics from between points of $q\cup F^*$ are in $Y$ to begin with. Thus, $q\cup F$ is $\kappa$-quasiconvex.
\end{proof}

The two lemmas below relate how close together a ray $q\in \bot(F)$ and $c$ can be, where $c$ is the ray based at $x_0$ asymptotic to $r$.

\begin{lem}
	There is a constant $M=M(\kappa)>0$ such that for any ray $q\in \bot(F)$, $d(q(0),c(t))< M$ where $c$ is the ray asymptotic to $q$ based at $x_0$ and $t=d(x_0,q(0))$.
\end{lem}

\begin{proof}
	Let $\beta\colon [0,t]\to F$ be the geodesic from $x_0$ to $q(0)$. Using the quasiconvexity result above, there is a $\kappa>0$ such that $c$ is contained in the $\kappa$-neighborhood of $q\cup F$. Therefore, there are points $s\in[0,\infty)$, $x\in q$, $y\in F$ such that $d(c(s),x)\leq \kappa$ and $d(c(s),y)\leq \kappa$. Using the triangle inequality, $d(x,y)\leq 2\kappa$. Since $x\in q$, $q$ is orthogonal $F$ and $y\in F$, then $d(x,q(0))\leq 2\kappa$ as well and so $d(c(s),q(0))\leq 3\kappa$. 
	
	The triangle $\triangle(x_0,c(s),q(0))$ has edges with side lengths $t$, $s,$ and at most $3\kappa$, so $t\in [s-3\kappa,s+3\kappa]$. Thus $|t-s|\leq 3\kappa$ and so \[d(c(t),q(0)) \leq d(c(t),c(s))+ d(c(s),q(0)) \leq 3\kappa+3\kappa.\] Setting $M=6\kappa$ is the desired constant.
\end{proof}

The next lemma relates rays in $\bot(F)$ and their geodesic representatives, based on how close the former ray's starting point is to a ray in $F$.

\begin{lem} \label{delta(epsilon,M)}
	Fix $\epsilon>0$ and let $M=M(\kappa)$ from above. Then there is a constant $\delta=\delta(\epsilon,M)$ such that for any $n>0$ and $\eta \in \bnd F$ the following holds: if $q\in \bot(F)$ and $q(0)\in U(\eta, n, \epsilon)$, then $c(n)\in U(\eta,n,\delta)$, where $c$ is the ray asymptotic to $q$.
\end{lem}

\begin{proof}
	Let $\beta\colon  [0,t]\to F$ be the geodesic from $x_0$ to $q(0)$. The lemma above shows that if $t=n$, then $d(c(t),q(0))\leq M$ so $d(c(t),\eta(n))\leq M+\epsilon$.
	
	If $n<t$, notice that $d(c(t),\eta(n)) \leq d(c(t),\beta(t))+ d(\beta(t),\eta(n)) \leq M+\epsilon$. Consider the triangle $\triangle(x_0,c(t),\eta(n))$. This has side lengths $t$, $n$, and at most $M+\epsilon$, so $t-n\leq M+\epsilon$. With the triangle inequality we have:	
	$$d(c(n),\eta(n)) \leq d(c(t),c(n)) + d(c(t), \eta(n)) \leq M+\epsilon + M + \epsilon$$
	Lastly, if $t<n$, then $d(\beta(t),\eta(n))\leq \epsilon$ and $d(\beta(t),c(t)) \leq M$ so \linebreak $d(\eta(n),c(t))\leq M+\epsilon$. Consider the triangle $\triangle(x_0,\eta(n),\beta(t)$, which has sidelengths $t$, $n$, and at most $\epsilon$. Therefore $n-t< \epsilon$. Applying the triangle inequalty we get:	
	$$d(c(n),\eta(n)) \leq d(c(t),c(n))+ d(c(t),\eta(n)) \leq \epsilon + M+\epsilon$$
	Therefore, letting $\delta = 2(M+\epsilon)$ covers all three cases.
\end{proof}

With these two lemmas, we can now prove that the projection maps are proper:

\begin{prop}
	The projection map $\pi_\Omega\colon \R \to \Omega$ is a proper map.
\end{prop}

\begin{proof}
	Let $K\subset \Omega$ be a compact set and consider 
	\[\pi_\Omega\inverse(K) = \{ (x,q): q\in \bot(F),~ q(\infty)\in K,~ d(x,q(0))\leq A\}\subset \R\]
	Let $C$ be the following set
	\[ C := \{ x\in F: \exists q\in \bot(F), ~q(\infty)\in K,~ d(q(0),x)\leq A\} \]
	We claim $C$ is compact and we will show this by showing that it is closed and bounded. 
	
	First, suppose it is unbounded. Then there is a sequence $x_n$ tending to infinity and so there is a sequence of rays $q_n$ with $q_n(0)$ tending to a point $\eta\in \bnd F$. Let $c_n$ be the geodesic ray based at $x_0$ asymptotic to $q_n$. Fix $\epsilon>0$ and consider $U(\eta,n,\epsilon)$. Such a neighborhood contains all but finitely many of $q_i(0)$. By Lemma \ref{delta(epsilon,M)}, $U(\eta,n,\delta)$ contains all but finitely many $c_i$ and therefore $c_i(\infty)\to \eta$. But $c_i(\infty)\in K$ and $K$ is compact, so $\eta\in K$, a contradiction. Therefore $C$ is bounded.
	
	Next, let $x_i\to x$ be a sequence of points in $C$. For each $x_i$, there is a a $q_i\in \bot(F)$ with $d(q_i(0),x_i)\leq A$. So for sufficiently large $i$, $q_i(0)\in \overline{B_A(x)}$. Applying Lemma \ref{isometric embeddings converge}, $q_i$ converge to $q\in \bot(F)$ pointwise. Therefore $q(0)\in \overline{B_A(x)}$. Since $q_i(\infty)\in K$ for all $i$ and $K$ is compact, then $q(\infty)\in K$ and so $q(0)\in C$ and so is $x$.
	
	Now let $(x_n,q_n)$ be a sequence in $\pi_\Omega\inverse(K)$. Since $q_n(0)\in C$ and $C$ is compact, by Lemma \ref{isometric embeddings converge}, after possibly passing to a subsequence, $q_n\to q$ pointwise. Similarly, $x_n\in C$, so after passing to a subsequence, $x_n$ converge as well and so $(x_n,q_n)$ can be made to converge to $(x,q)$.
	
	Lastly, to see that $q\in \bot(F)$, take any $y\in F$. Let $p_n=q_n(0)$ and $x_n=q_n(1)$. Since each $q_n$ is orthogonal to $F$, $\angle_{p_n}(x_,y)\geq \pi/2$. The map $(x,y,p)\to \angle_p(x,y)$ is upper semicontinuous (see \cite{BH99} Proposition II.3.3(1))), so $q\in \bot(F)$. This concludes the proof.
	
\end{proof}
\begin{prop}
	The projection map $\pi_F\colon \R \to F$ is a proper map.
\end{prop}

\begin{proof}
	Let $C\subset F$ be compact. Then we have
	\[ \pi_F\inverse (C) = \{ (x,q): q\in \bot(F), ~x\in C, ~d(q(0),x)\leq A\}\]
	Let $(x_,q_n)$ be a sequence in $\pi_F \inverse(C)$. Since $C$ is compact, so is $\overline{B_A(C)}$ and $q_i(0)\in \overline{B_A(C)}$ for each $i$. By Lemma \ref{isometric embeddings converge}, after possibly passing to a subsequence, $q_n\to q$ pointwise. Furthermore, $x_n\to x$ since $x_n\in C$. Therefore, $(x_n,q_n)\to (x,q)$.
	
		To see that $q\in \bot(F)$, take any $y\in F$. Let $p_n=q_n(0)$ and $x_n=q_n(1)$. Since each $q_n$ is orthogonal to $F$, $\angle_{p_n}(x_,y)\geq \pi/2$. The map $(x,y,p)\to \angle_p(x,y)$ is upper semicontinuous (see \cite{BH99} Proposition II.3.3(1))), so $q\in \bot(F)$ and thus $(x,q)\in \pi_F\inverse(C)$ making this set compact.
\end{proof}
Combining everything we get the following theorem:

\begin{thm}
	Let $P$ be the stabilizer of a flat $F$ of $X_v$. Then $P$ acts properly and cocompactly on $\Omega:= \bnd Y \setminus \bnd F^*$, where $F^*$ is the maximal dimensional flat in $Y$ which contains $F$.
\end{thm}

More importantly, the work above gives the following association between compact sets of $F$ and those of $\Omega$:

\begin{corollary}
	\label{Compact Set Association}
	Given a compact set $K\subset \Omega$, there exists a compact set $C\subset F \subset F^*$ associated to $K$  such that for each $\eta \in K$, there is a ray $c\colon [0,\infty) \to Y$ such that $c\in \bot(F)=\bot(F^*)$, $c(0)\in C$ and $c(\infty)=\eta$. Furthermore, the set $C$ can be chosen $P$-equivariantly in the sense that if $C$ is associated to $K$, then $pC$ is associated to $pK$.
\end{corollary}

\begin{proof}
	This follows immediately from the proofs of the projection maps being proper. For any $K\subset \Omega$, take $C$ to be $\pi_F(\pi_\Omega\inverse(K))$, which is compact as seen above and this association is $P$-equivariant since $\R$ is.
\end{proof}

This corollary is vital for relating neighborhoods of points in $F$ to neighborhoods of points in $\Omega$, which will ultimately allow us to show that $\bnd Y$ is locally connected.

\section{Local Connectivity in Flats} \label{section:local conn at flats}
The goal of this section to show that $\bnd Y$ is locally connected at each of the remaining points. Once that is done, we can combine everything we have done to prove Theorem \ref{thm main}. The decomposition theory gives that for each $\xi\in \bnd Y$, if $\xi\notin \bnd F^*$ for any flat $F^*$ of $Y$, then $\bnd Y$ is locally connected at $\xi$. Additionally, if $\beta\in \bnd F^*$ and $\beta \notin \bnd X_v$, then $\bnd Y$ is locally connected at $\beta$ since $\bnd F^*$ is a sphere. Therefore all that is left to show is that if $\alpha \in \bnd X_v \cap \bnd F^*= \bnd F$, then $\bnd Y$ is locally connected at $\alpha$.

Instead of showing local connectivity, we show that if $\alpha\in \bnd F$, then $\bnd Y$ is weakly locally connected at $\alpha$. The distinction between locally connected and weakly locally connected is toward the end of the section, Definition \ref{def:wlc}.

Many of the ideas in this section come from \cite{HR17}. The authors work in the setting where $G$ is a $\cat(0)$ group with isolated flats with a trivial maximal peripheral splitting and ultimately conclude that $\bnd G$ is  locally connected at every point. We show here that the ideas work in our setting as well.

The approach is to use the association established in Corollary \ref{Compact Set Association} to relate neighborhoods in $\overline{F}$ to neighborhoods of $\overline{\Omega}$. Recall $F$ is the subflat of $F^*$ such that $\bnd F \subset \bnd X_v$. Before establishing this connection, we need to understand $\overline{\Omega}$.

\begin{lem}
	$\overline{\Omega}=\Omega \cup \bnd F$.
\end{lem}

\begin{proof}
	First, we will show that $\Omega \cup \bnd F$ is closed. We know that $\bnd F^* $ and $\bnd F$ are both closed in $\bnd Y$. Therefore $\bnd F^* \setminus \bnd F$ is open. We can also see that 
	$$\bnd Y \setminus (\Omega \cup \bnd F) = \bnd F^* \setminus \bnd F$$
	Therefore $\Omega \cup \bnd F$ is closed since its complement in $\bnd Y$ is open and so $\overline{\Omega}\subset \Omega \cup \bnd F$.
	
	Now suppose $\alpha\in \bnd F$, then we want to show that $\alpha\in \overline{\Omega}$. We will do this by creating a sequence of points in $\Omega$ which limit to $\alpha$. Fix $x_0\in F$ and let $c\colon [0,\infty)\to Y$ be the ray representing $\alpha$ and let $C$ be a fundamental domain for the action of $P$ on $F$. Using the $P$-action, let $p_i\cdot C$ cover the ray $c$. For each $p_i$, let $q_i\in \bot(F)$ such that $q_i(0)\in p_i\cdot C$. Let $r_i$ be the geodesic ray based at $x_0$ asymptotic to $q_i$.
	
	Fix $\epsilon>0$ and notice that since the translates of $C$ cover $c$, then for all but finitely many $i$, $q_i(0)\in U(c,n,\epsilon)$. Applying Lemma \ref{delta(epsilon,M)}, we have  \[r_i(n)\in U(c,n,\epsilon+\delta)\]for all but finitely many $i$. This means that $r_i(\infty)\in U(c,n,\epsilon+\delta)$ for all but finitely $i$ and thus $r_i(\infty)\to \alpha$ as $n\to \infty$, as desired.

\end{proof}

Next we restate a result of Haulmark's to emphasize the connection between rays orthogonal to the flat $F$ and their corresponding endpoints:

\begin{lem} [\cite{Ha17}] \label{Haulmark Quasi Convex}
	Let $\kappa$ be the CAT(0) with isolated flats constant as in Theorem \ref{Quasi-convex Constant}. Suppose $c'\in \bot(F)$ and $c$ is some ray contained in $F$. If $c'(0)\in U(c,r,D)$ for some constants $r,D$, then $c'(\infty)\in U(c,r,D+\kappa)$. Conversely, if $c'(\infty)\in U(c,r,D)$, then $c'(0)\in U(c,r,D+\kappa)$.
\end{lem}

We make use of the following result of Bestvina:

\begin{prop}[\cite{Bes96}]
	Let $H$ be a group acting geometrically on a $\cat(0)$ space $Z$. Let $C\subset Z$ be a compact set. The $H$-translates of $C$ form a null family in the compact space $\overline{Z}$.
\end{prop}

\begin{prop}
	If $K\subset \Omega$ is compect, then the collection of $P$-translates of $K$ is a null family in $\overline{\Omega}$, where $P$ is the stabilizer of the flat $F\subset X_v$.
\end{prop}

\begin{proof}	
	Let $K\subset \Omega$ be compact and let $C\subset F$ be the compact set associated to $K$ from Corollary \ref{Compact Set Association}. Fix a positive constant $D>0$ and let $\kappa$ be the constant from Theorem \ref{Quasi-convex Constant}. From the Remark \ref{remark: null family condition}, it suffices to show that for each $r<\infty$, only finitely many $P$-translates of $K$ are not contain in any set of the form $U(\cdot, r, D)$. 
	
	Using the Bestvina result above, $P^*$-translates of $C$ form a null family in $\overline{F^*}$. Since $P\subset P^*$ and $\overline{F}\subset \overline{F^*}$, then the $P$-translates of $C$ form a null family in $\overline{F}$. So only finitely many of the $P$-translates of $C$ are not in set of the form $U(\cdot, r,D+\kappa)$. For any $p\in P$, if $pK$ lies in a set of the form $U(\cdot, r, D)$, then $pC$ lies in a set of the form $U(\cdot, r, D+\kappa)$ by Lemma \ref{Haulmark Quasi Convex}. Therefore the $P$-translates of $K$ form a null family in $\overline{\Omega}$.
\end{proof}

The association between compact sets of $F$ and those of $\Omega$ established in Corollary \ref{Compact Set Association} can be improved even further: the compact sets chosen can be connected fundamental domains for the action of $P$. To see this is connected, first observe the following:

\begin{prop} \label{Connected Fundamental Domain}
	Let $(G,\bP)$ be a CAT(0) group with isolated flats and peripheral structure $\bP$. Let $(G_v,\bP_v)$ be a component vertex of $G$'s maximal peripheral splitting with the peripheral structure coming from $G$. Let \linebreak $\rho \in \partial(G_v, \bP_v)$ be a parabolic point stabilized by $P$. Then there exists a connected, compact fundamental domain $C$ for the action of $P$ on $\partial (G_v, \bP_v) \setminus \{\rho\}$. 
\end{prop}

\begin{proof}
	
	Since $(G_v,\bP_v)$ is a component vertex of a 1-ended CAT(0) group with isolated flats, we know that $\partial (G_v, \bP_v)$ is connected and locally connected and therefore path connected. Notice that $\pi(F)=\pi(F^*)=\rho$. We also know that since $G_v$ is a component vertex of a maximal peripheral splitting, then $G_v$ does not split further relative to $\bP_v$. Therefore $\rho$ is not a global cut point of $\partial (G_v, \bP_v)$. These facts follow from Theorem \ref{thm: peripheral splitting}.
	
	Since $\rho$ is not a global cut point and $\partial (G_v, \bP_v)$ is path connected, then $\partial (G_v, \bP_v)\setminus \{\rho\}$ is path connected as well. We know that $P$ acts cocompactly on $\partial (G_v, \bP_v)\setminus \{\rho\}$. Let $C_0$ be a compact fundamental domain for this action. Our goal is to show that $C_0$ is contained in a connected fundamental domain $C$. 
	
	Let $d= d(C_0, \rho)$ in $\partial (G_v, \bP_v)$. We can cover $C_0$ with finitely many open connected sets of diameter at most $d/2$. The union of the closures of these sets form a compact set $C_1\subset \partial (G_v, \bP_v)$ containing $C_0$ and having only finitely many connected components. By our choice of $d$, $C_1 \subset \partial (G_v, \bP_v) \setminus \{\rho\}$. We can then join the finitely many connected components of $C_1$ with finitely many paths in $\partial (G_v, \bP_v) \setminus \{\rho\}$, giving us a compact connected fundamental domain $C$.
	
\end{proof}

Using this, we can find a connected fundamental domain for $P$'s action on $\Omega$.

\begin{prop}\label{Fundamental Domain on Omega}
	There is a compact, connected fundamental domain $K\subset \Omega$ for the action of $P$ on $\Omega$. Furthermore, if $S$ is any finite generating set for $P$, then $K$ can be chosen such that $sK\cap K \neq \emptyset$ for each $s\in S$. 
\end{prop}

\begin{proof}
	We will use the decomposition map $\pi\colon  \bnd Y \to \bnd (G_v,\bP_v)$ and then pull back the compact, connected fundamental domain for $P$'s action on $\bnd(G_v,\bP_v) \setminus \{\rho \}$ to get a fundamental domain on $\Omega$.
	
	Let $C_0$ be a connected fundamental domain for the action of $P$ on \linebreak $\bnd (G_v, \bP_v) \setminus \{\rho\}$ as in Proposition \ref{Connected Fundamental Domain}. Without loss of generality, we may increase the size of $C_0$ such that $C_0\cap sC_0 \neq \emptyset$ for all $s\in S$. We can then repeat the process from Proposition \ref{Connected Fundamental Domain} to construct a connected $C$.
	
	Since the map $\pi\colon \bnd Y \to \bnd (G_v, \bP_v)$ defines an upper semicontinuous monotone decomposition, then we know that $\pi \inverse(C)$ is compact and connected. 	The map $\pi$ is $G$-equivariant, and therefore since $C$ is a fundamental domain, $K$ is as well. Since $C$ satisfies the desired intersection property, then $K$ does as well.
\end{proof}

Using the association between subsets of $F$ and those of $\Omega$, we can establish a way to determine when subsets of $\Omega$ are connected.

\begin{prop} \label{K and C}
	There exists compact, connected fundamental domains $K$ and $C$ for the action of $P$ on $\Omega$ and $F$, respectively, such that the following property holds: Let $\P$ be any subset of $P$. Then:
	
	\begin{center}
		If $\bigcup\limits_{p\in \P} pC$ is connected in $F$, then $\bigcup\limits_{p\in \P} pK$ is connected in $\Omega$.
	\end{center}
\end{prop}

\begin{proof}
	Choose a compact, connected fundamental domain $C$ for the action of $P$ on $F$. Let $S$ be the set of $p\in P$ such that $pC\cap C\neq \emptyset$. Then $S$ is a finite generating set for $P$. Use Proposition \ref{Fundamental Domain on Omega} to pick $K$ such that $sK\cap K \neq \emptyset$ for all $s\in S$. This implies the following, which implies the desired conclusion:	
	\begin{center}
		If $C\cap pC \neq \emptyset$, then $K\cap pK \neq \emptyset$
	\end{center}
\end{proof}

We will now fix the compact, connected fundamental domains $K$ and $C$ for $P$'s action on $\Omega$ and $F$, respectively. Using Corollary \ref{Compact Set Association}, let $c'\colon [0,\infty)\to Y$ be a geodesic ray which meets $F$ orthogonally. Using the $P$-action, we may assume that $c'(0)\in C$ and $c'(\infty)\in K$. We will call $c'(0)=q_0$ and use $q_0$ as the basepoint for the cone topology.

Let $\alpha\in \bnd F$. Our method is two-fold. First,  associate connected neighborhoods of $\alpha$ in $\overline{F^*}$ to neighborhoods of $\alpha$ in $\overline{\Omega}$, which by Proposition \ref{K and C} will be connected. Then, given a neighborhood of $\alpha$ in $\overline{\Omega}$, find a connected neighborhood of $\alpha$ in $\overline{F^*}$ to associate inside of the neighborhood in $\overline{\Omega}$. 

To make this possible, we need to find suitably nice neighborhoods in $\overline{F^*}$.

\begin{defn}
	Let $Z$ be a $\cat(0)$ space and $\xi\in \bnd Z$. A neighborhood $\overline{N}$ of $\xi$ in $\overline{Z}$ is \textit{clean} if the following conditions hold:
	
	\begin{enumerate}
		\item $\overline{N}$ is connected.
		\item $N=\overline{N}\cap Z$ is connected, and
		\item Every point of $\Lambda:=\overline{N}\cap \bnd Z$ is a limit point of $N$.
	\end{enumerate}
\end{defn}

Hruska and Ruane prove  (\cite{HR17}, Proposition 3.3) that for any $\xi\in \bnd Z$, $\xi$ has a local base of clean connected neighborhoods. We show below that we can pick this neighborhood basis to be clean and connected both in $\overline{F}$ and $\overline{F^*}$. 

\begin{lem}
	There is a local base of neighborhoods of $\alpha$ in $\overline{F^*}$ which are clean and connected in $\overline{F}$ and in $\overline{F^*}$. 
\end{lem}

\begin{proof}
	Fix $q_0$ as the basepoint for the cone topology and let $c\colon [0,\infty)\to F^*$ be the ray based at $q_0$ representing $\alpha$. It is shown in \cite{HR17} that basic open sets, $U(c,r,D)$, are clean and connected in $\overline{F^*}$.  
	
	First note that since $c(0)\in F$, $c(\infty)\in \bnd F$ and $F$ is convex in $F^*$, then $U(c,r,D)\cap \overline{F}$ is non-empty. Denote $N_F= F\cap U(c,r,D)$. 
	
	A point $p\in F$ is in $N_F$ if and only if the geodesic segment from $c(0)$ to $p$ intersects $B_D(c(r))$. Let $\gamma\colon [0,t]\to F$ be the geodesic segment $[c(0),p]$, then we can choose $s$ so that $d(\gamma(s),c(r))<D$. It follows that the entire geodesic segment $[\gamma(s),\gamma(t)]\subset U(c,r,D)$, so every point of $N_F$ is contained in a connected subset of $N_F$ which intersects the connected ball $B_D(c(r))$ and thus $N_F$ is connected. All of $U(c,r,D)$ is contained in the closure of $N_F\subset \overline{F}$, and therefore is connected as well. Since $U(c,r,D)$ is contained in the closure of $N$, the limit set requirement is satisfied as well.
	
\end{proof}

For a given clean, connected neighborhood $\overline{N}=N\cup \Lambda$, let \linebreak $\P:=\{p\in P: pC\cap N\neq \emptyset\}$. Since $q_0\in F$, $\P$ is non-empty and by Proposition \ref{K and C}, $\bigcup\limits_{p\in \P} pC$ is connected. 

\begin{defn}
	Let $\alpha\in \bnd F$ and $\overline{N}$ be a clean, connected neighborhood with $\overline{N}\cap F = N$ and $\overline{N}\cap \bnd F = \Lambda'$. Let $\P := \{p\in P: pC\cap N\neq \emptyset\}$. Then the \textit{neighborhood in $\overline{\Omega}$ of $\alpha$ associated to $\overline{N}$} is $\overline{Z}= \left( \bigcup\limits_{p\in \P} pK\right) \cup \Lambda'$. 
\end{defn} 

We now need to show three things: $\overline{Z}$ is a neighborhood of $\alpha$, it is connected, and when $\overline{N}$ can be chosen to be `small,' then $\overline{Z}$ is `small' as well.

\begin{lem} \label{Z is a nieghborhood}
	The set $\overline{Z}$ associated to $\overline{N}$ is a neighborhood of $\alpha$ and is connected.
\end{lem}

\begin{proof}
	First, we show that $\overline{Z}$ is a neighborhood of $\alpha$. Let $c\colon [0,\infty)\to F^*$ be the ray based at $q_0$ representing $\alpha$ and fix some $D>0$. Since $\overline{N}$ is a neighborhood of $\alpha$, we can choose $R$ large enough so that $U(c,R,D+\kappa)\cap \overline{F^*}$ is contained in $\overline{N}$, and therefore $U(c,R,D)\cap \overline{F^*}$ is contained in $\overline{N}$ as well. The $P$-translates of $K$ form a null family, so there is a neighborhood $V$ of $\alpha$ with $V\subset U(c,R,D)$ such that if $pK\cap V \neq \emptyset$, then $pK\subset U(c,R,D)$.
	
	Now it suffices to show that $V\subset \overline{Z}$. For $\eta\in V$, either $\eta\in \Omega$ or $\eta\in \bnd F$. In the former case, $\eta\in pK$ for some $p\in P$ and so $pK\subset U(c,R,D)$. In particular, $p\cdot q_\infty\in U(c,R,D)$ and by Lemma \ref{Haulmark Quasi Convex}, $p\cdot q_0\in U(c,R,D+\kappa)$. Since $q_0\in C$ and $U(c,R,D+\kappa)\cap \overline{F^*}\subset \overline{N}$, then $pC\cap N \neq \emptyset$. Therefore, $\eta\in \overline{Z}$ by the construction.
	
	In the latter case, $\eta\in \bnd F\cap V$, so
	
	$$\eta \in U(c,R,D)\cap \bnd F \subset U(c,R,D+\kappa)\cap \bnd F \subset \overline{N}\cap \bnd F = \Lambda'\subset \overline{Z}$$
	Therefore, $V\subset \overline{Z}$ and so it is a neighborhood of $\alpha$ in $\overline{\Omega}$.
	
	Now we must show that $\overline{Z}$ is connected. By Proposition \ref{K and C}, $\overline{Z}\cap \Omega$ is connected. Therefore it suffices to show that every point of $\Lambda'$ is a limit point of $\overline{Z}\cap \Omega$. Let $\xi\in \Lambda'$. Since $\overline{N}$ is clean, $\xi$ is the limit of a sequence $x_i\in N$. Each $x_i\in p_iC$ for some $p_i\in P$. Fix $\epsilon>0$ and note that since $x_i\to \xi$, $x_i\in U(\xi,R,\epsilon)$ for all but finitely many $i$. Let $A$ be the diameter of $C$, then $p_i\cdot q_0\in U(\xi,R,\epsilon+A)$ for all but finitely many $i$, and so by Lemma \ref{delta(epsilon,M)}, $p_i\cdot q(R)\in U(\xi,R,\epsilon+A+\delta)$ for all but finitely many $i$. Therefore, $p_i\cdot q(\infty)\in U(\xi,R,\epsilon+A+\delta)$ and thus $p_i\cdot q(\infty)$ converge to $\xi$. Since $p_i\cdot C \cap N \neq \emptyset$, then $p_i\cdot K\cap \overline{Z} \neq \emptyset$ and so $\xi$ is the limit of points in $\overline{Z}\cap \Omega$ as desired.
\end{proof}

\begin{lem} \label{Z small}
	The neighborhood $\overline{Z}$ associated to $\overline{N}$ can be made arbitrarily small by choosing $\overline{N}$ to be a sufficiently small neighborhood of $\alpha$ in $\overline{F^*}$
\end{lem}

\begin{proof}
	Let $U$ be a neighborhood of $\alpha$ in $\overline{\Omega}$. Since the $P$-translates of $K$ form a null family. There is a neighborhood $V\subset U$ such that if $pK\cap V$ is non-empty, then $pK\subset U$. Since $V$ is a neighborhood of $\alpha$, there are constants $R>0$ and $D>0$ such that $U(c,R,D+\kappa)\subset V$ and $U(c,R,D)\cap \bnd F\subset V$. Then, by Lemma \ref{Haulmark Quasi Convex}, $U(c,R,D)$ is a neighborhood of $\alpha$ in $\overline{F^*}$ such that if $p\cdot q_0\in U(c,R,D)$, then $p\cdot q_\infty \in U(c,R,D+\kappa)$ and so $pK\subset U$. Furthermore, since the $P$-translates of $C$ form a null family, there is a $W\subset U(c,R,D)$ such that if $pC\cap W$ is non-empty, then $pC\subset U(c,R,D)$. Finally, there is a clean, connected neighborhood $\overline{N}$ within $W$.
	
	It remains to show that the associated neighborhood $\overline{Z}$ is contained in $U$. Let $\eta\in \overline{Z}$. By our choice of $U(c,R,D)$, if $\eta\in \bnd F$, then $\eta \in V$ and therefore in $U$. If $\eta \in \Omega$, then $\eta\in pK$ for some $p$, so $pK\cap \overline{Z}$ is non-empty and thus $pC\cap N$ is non-empty. This means that $pC\cap W\neq \emptyset$ so $pC\subset U(c,R,D)$ and thus $pK$ intersects $V$ non-trivially. By the choice of $V$, $pK\subset U$ and so $\eta\in U$ as desired. Thus $\overline{Z}$ is contained in $U$.

\end{proof}

Recall the definitions of weakly locally connected and locally connected:

\begin{defn}[Weakly locally connected]\label{def:wlc}
	Let $M$ be a topological space. $M$ is \textit{weakly locally connected} at $m\in M$ if for every open set $V$, there  is a connected (and not necessarily open) set $N\subset V$ with $m$ in the interior of $N$.
	
	We say a space $M$ is weakly locally connected if it is weakly locally connected at every point.
\end{defn}

\begin{defn}[Locally connected]
	Let $M$ be a topological space. $M$ is \textit{locally connected} at $m\in M$ if for every open set $V$ containing $m$, there is a connected open set $U$ with $m\in U \subset V$. 
	
	We say $M$ is locally connected if it is locally connected at every point.
\end{defn}

In Lemma \ref{Z is a nieghborhood} and Lemma \ref{Z small}, we do not know if $\overline{Z}$ is open or not, so we cannot prove local connectivity in $\bnd Y$. Instead, we show the following:
\begin{prop}\label{Y is wlc}
	If $\alpha\in \bnd Y \cap \bnd F$, then $\bnd Y$ is weakly locally connected at $\alpha$.
\end{prop}

\begin{proof}
	Let $U$ be a neighborhood of $\alpha$ in $\bnd Y$. Since $\bnd Y = \overline{\Omega} \cup \bnd F^*$, then $U=U_1\cup U_2$, with $U_1$ a neighborhood of $\alpha$ in $\overline{\Omega}$ and $U_2$ a neighborhood of $\alpha$ in $\bnd F^*$. Combining Lemma \ref{Z is a nieghborhood} and Lemma \ref{Z small}, there is a $V_1\subset U_1$ which is connected and contains $\alpha$. Since $\bnd F^*$ is a sphere, it is locally connected, so there is some $V_2\subset U_2$ which is connected and contains $\alpha$. Lastly, since $V_1\cap V_2$ is non-empty, then their union is connected as well. Thus $V_1\cup V_2\subset U$ is a connected neighborhood of $\alpha$.
\end{proof}

Combining all of this work leads to the following theorem:

\begin{thm}\label{Y is pconn}
	Let $(G,\bP)$ be a $\cat(0)$ group with isolated flats acting geometrically on $X$, let $G_v$ be a component vertex in its maximal peripheral splitting, acting geometrically on $X_v\subset X$ and let $Y$ be as in Definition \ref{defn Y}. Then $\bnd Y$ is path connected. 
\end{thm}

\begin{proof}
	
	It suffices to show that $\bnd Y$ is both connected and weakly locally connected at every point. By Theorem 27.16 in \cite{Wil70}, a space which is weakly locally connected at every point is locally connected at every point. Also by Theorem 31.2 in \cite{Wil70}, a space which is compact, connected, locally connected and metrizable is locally path connected, which implies globally path connected (by \cite{Wil70}, Theorem 27.5). Since boundaries of proper metric spaces are compact and metrizable, all we need to show is that $\bnd Y$ is connected and weakly locally connected at each point. 
	
	Connectedness follows from Theorem \ref{monotone usc decomp of Y}. Since the map $\pi\colon \bnd Y \to \bnd (G_v,\bP_v)$ is monotone and $\bnd(G_v,\bP_v)$ is connected, then $\bnd Y$ is connected as well.
	
	For weakly locally connected, there are three types of points in $\bnd Y$ to consider: points in $\bnd X_v$ and not a flat, points just in a flat, and points in both $\bnd X_v$ and $\bnd F$ for some flat $F$. By Corollary \ref{Y non-flat points lconn}, if $x\in \bnd Y$ and $x\notin \bnd F$ for any flat $F$, then $\bnd Y$ is locally connected at $x$ and therefore weakly locally connected at $x$. If $\xi\in \bnd F^*$ for some flat $F^*$ and $\xi \notin \bnd X_v$, then $\bnd Y$ is locally connected at $\xi$ and so weakly locally connected at $\xi$. This is because sufficiently small neighborhoods of $\xi$ only contain points in $\bnd F^*$, which is locally connected. Lastly, if $\alpha\in \bnd X_v \cap \bnd F$ for some flat $F$, then by Proposition \ref{Y is wlc}, $\bnd Y$ is weakly locally connected at $\alpha$.
\end{proof}

We are almost ready to combine all this work to conclude that $\cat(0)$ groups with isolated flats have path connected visual boundaries. But first we need to show that these groups satisfy the lonely condition.

\begin{lem} \label{IFP blocks are lonely}
	Suppose $G$ is a 1-ended $\cat(0)$ group with isolated flats acting geometrically on a $\cat(0)$ space $X$. Let $\B$ be the block decomposition coming from the maximal peripheral splitting of $\G$, then all irrational rays are lonely.
\end{lem}

\begin{proof}
	For any flat $F$, let $\gamma_F\colon [0,b]\to X$ be the geodesic segment from $x_0$ to the closest point of $F$. Since $F$ is convex, a closest point exists and is unique. By Theorem \ref{Quasi-convex Constant}, $\gamma_F\cup F$ is $\kappa$-quaisconvex. In particular, every geodesic starting at $x_0$ which intersects $F$ must be the in $\kappa$-neighborhood of $\gamma_F(b)$. 
	
	Let $c,c'$ be rays in $X$ based at $x_0$ with the same, infinite itinerary. Since these rays have the same itinerary, they pass through the same sequence of flats, and must pass through infinitely many flats. But then the rays are $2\kappa$-close infinitely often. By the convexity of the distance metric, the rays must be $2\kappa$-close always, and therefore $c$ and $c'$ are asymptotic.
	
\end{proof}

\begin{thm} \label{thm main} 
	All 1-ended $\cat(0)$ groups with isolated flats have path connected visual boundaries.
\end{thm}

\begin{proof}
	Let $(G,\bP)$ be a $\cat(0)$ group with isolated flats. It suffices to show there is a space $X$ on which $G$ acts which has path connected visual boundary. Let $\G$ be the maximal peripheral splitting for $(G,\bP)$, $X$ be the tree-of-spaces on which $G$ acts geometrically, and $\B$ the block decomposition coming from this peripheral splitting.
	
	Each peripheral vertex group in the splitting is virtually $\Z^n$ for some $n$, so has path connected visual boundary. If each component vertex has path connected visual boundary as well, then by Theorem \ref{combo} $\bnd X$ is path connected. By Lemma \ref{IFP blocks are lonely}, the irrational rays are all lonely, so we can apply the theorem.
	
	If there is a peripheral vertex which does not have path connected visual boundary, then by Theorem \ref{Y is pconn}, $\partial_R X$ is path connected. This is true because
	$$\partial_R X = \bigcup_{v\in \C} \bnd Y_v$$
	where $\C$ is the collection of component vertices in the tree $\T$ for the splitting $\G$. Therefore, we can apply Proposition \ref{lonely plus rational is pconn} to conclude that $\bnd X$ is path connected.
\end{proof}

\bibliographystyle{alpha}
\bibliography{benzvibib}	
\end{document}